\documentclass[11pt]{article}
\usepackage[french, english]{babel}
\usepackage[utf8]{inputenc}
\usepackage{amsmath,amsfonts,amssymb,amsthm,mathrsfs}
\usepackage{minitoc}
\usepackage{bbm}
\renewcommand{\leq}{\leqslant}
\renewcommand{\geq}{\geqslant}
\usepackage[mono=false]{libertine}
\useosf
\usepackage{mathpazo}

\usepackage{wasysym}

\usepackage[dvipsnames]{xcolor}
\usepackage[a4paper,vmargin={3.5cm,3.5cm},hmargin={2cm,2cm}]{geometry}
\usepackage[font=sf, labelfont={sf,bf}, margin=1cm]{caption}
\usepackage{graphicx,graphics}
\usepackage{epsfig}
\usepackage{latexsym}
\usepackage{stmaryrd}
\usepackage{ae,aecompl}

\usepackage[english]{babel}
 \usepackage[colorlinks=true]{hyperref}
\usepackage{pstricks}
\usepackage{enumerate}
\usepackage{tikz}
\usepackage{fontawesome}
\usepackage{soul}
\usetikzlibrary{arrows,automata}
\usepackage{pdfpages}

\renewcommand{\P}{\mathbb{P}}
\newcommand{\E}{\mathbb{E}}
\DeclareFixedFont{\beaupetit}{T1}{ftp}{b}{n}{2cm}

\newcommand{\Dd}{\mathcal{R}}
\newcommand{\D}{\mathbb{D}}

\newtheorem{theorem}{Theorem}[]

\newtheorem{proposition}[]{Proposition}

\newtheorem{lemma}[]{Lemma}

\theoremstyle{definition}

\newtheorem*{remark}{Remark}

\newcommand{\var}{\mathrm{Var}}
\newcommand{\eps}{\varepsilon}

\usepackage{todonotes}

\newcommand{\N}{\mathbb{N}}
\newcommand{\R}{\mathbb{R}}

\newcommand{\Ff}{\mathcal{F}}

\newcommand{\At}{\widetilde{A}}
\newcommand{\Bt}{\widetilde{B}}
\newcommand{\Ct}{\widetilde{C}}
\newcommand{\Tt}{\widetilde{\theta}}
\newcommand{\sx}{\mathscr{X}}
\newcommand{\sv}{\mathscr{V}}

\newcommand{\sz}{\mathscr{Z}}
\renewcommand{\ss}{\mathscr{S}}
\newcommand{\os}{\overline{s}}
\newcommand{\ox}{\overline{x}}
\newcommand{\oPhi}{\overline{\Phi}}
\newcommand{\GG}{\mathbb{G}}
\newcommand{\e}{\mathrm{e}}
\newcommand{\simp}{\mathrm{Simp}_{ \to k_0(\eta)}}
\newcommand{\simpbis}{\mathrm{Simp}_{ k_0(\eta) \to \theta^n}}
\newcommand{\KSC}{\mathrm{KSCore}}

\title{\textsc{The critical Karp--Sipser core of Erd\H{o}s--R\'enyi random graphs}}
\author{
Thomas \textsc{Budzinski}\thanks{ ENS de Lyon and CNRS.\hfill  \href{mailto:thomas.budzinski@ens-lyon.fr}{\texttt{thomas.budzinski@ens-lyon.fr}}}
\qquad\&\qquad
Alice \textsc{Contat}\thanks{Universit\'e Sorbonne Paris Nord.\hfill  \href{mailto:alice.contat@math.cnrs.fr}{\texttt{alice.contat@math.cnrs.fr}}}
}

\date{}
\begin{document}
\maketitle 

\abstract{The Karp--Sipser algorithm consists in removing recursively the leaves as well their unique neighbours and all isolated vertices of a given graph. The remaining graph obtained when there is no leaf left is called the Karp--Sipser core. When the underlying graph is the classical sparse Erd\H{o}s--Rényi random graph $ \mathrm{G}[n, \lambda/n]$, it is known to exhibit a phase transition at $\lambda = \mathrm{e}$. We show that at criticality, the Karp--Sipser core has size of order $n^{3/5}$, which proves a conjecture of Bauer and Golinelli. We provide the asymptotic law of this renormalized size as well as a description of the distribution of the core as a graph. Our approach relies on the differential equation method, and builds up on a previous work on a configuration model with bounded degrees.}

\section{Introduction}

\paragraph{The Karp--Sipser core.}
The Karp--Sipser algorithm on a graph $\mathfrak{g}$ consists in removing recursively the isolated vertices and the leaves of $ \mathfrak{g}$ together with their unique neighbours until all vertices have degree at least $2$. The subgraph that is obtained once there is no leaf and no isolated vertex left is called the \emph{Karp--Sipser core} of $\mathfrak{g}$ and we denote it by $ | \KSC (\mathfrak{g})|$.

The initial motivation of Karp and Sipser \cite{karp1981maximum} to study this algorithm is that it provides a computationally efficient way to find a large independent subset in a graph, although finding the largest such set is NP-hard. Indeed, the leaves and isolated vertices removed during the procedure form an independent set, meaning that no pair of these vertices are neighbours. Moreover, this is an optimal strategy in the sense that there exists an independent set of $\mathfrak{g}$ with maximal size which contains all these leaves and isolated vertices. Since then, this procedure has been extensively studied on random graph models, with applications to several different domains such as the rank of the adjacency matrix~\cite{bordenave2010rank} or quantum percolation phase transitions~\cite{bauer2001exactly}.

\paragraph{The Karp--Sipser core of random graphs.}
In their pioneering  paper \cite{karp1981maximum}, Karp and Sipser proved that the KS core of an Erd\H{o}s--R\'enyi random graph underlies an interesting phase transition in the sparse regime. More precisely, let $ \mathrm{G}[n,p]$ be the usual Erd\H{o}s--R\'enyi random graph, that is a random graph whose set of vertices is $ \{ 1, 2, \dots, n\}$ and containing each of the $ \binom{n}{2}$ possible edges with probability $p$ independently. Let us fix $\lambda>0$. Then 
$$ \frac{| \KSC (  \mathrm{G}[n,\lambda/n])|}{n} \xrightarrow[n\to\infty]{(d)} C( \lambda),$$
where $ C( \lambda)$ is a nonnegative constant and $ C ( \lambda) = 0$ if and only if $ \lambda \leq \mathrm{e}=2,7182...$. 

This result was later refined by Aronson, Frieze and Pittel \cite{AFP97} who proved that in the subcritical regime (i.e.\ when $ \lambda < \mathrm{e}$), the size of the Karp--Sipser core of $ \mathrm{G}[n, \lambda /n]$ converges in law towards an (explicit) sum of Poisson random variables as $n$ goes to infinity. Similar results have been obtained for more general configuration models \cite{bohman2011karp,jonckheere2021asymptotic}. 
In another direction, a central limit theorem for the size of the matching created by the algorithm (using an extended version that keeps running even when there is no leaf left) in the supercritical phase was obtained by Kreacic~\cite{kreacic2017some}. This was recently extended to the subcritical and critical regimes by $\lambda>\mathrm{e}$ Glasgow, Kwan, Sah and Sawhney~\cite{glasgow2024central}.

\paragraph{The critical regime.} Understanding the Karp--Sipser core in the critical regime is even more challenging. Very recently, the analog question was considered by the authors and Curien~\cite{BCC22} on a more tractable configuration model of graph with only vertices of degree $1$, $2$ and $3$. It was shown that for this configuration model, the Karp--Sipser core has size of order $ n^{3/5}$ in the critical regime, which confirmed a prediction of Bauer and Golinelli~\cite{bauer2001core}. Moreover, the KS core is mostly made of vertices of degree $2$ and contains only about $n^{2/5}$ vertices of degree $3$. The goal of this paper is to study Erd\H{o}s--R\'enyi random graphs $ \mathrm{G}[n,p]$ in the critical regime, i.e.\ when $p = \mathrm{e}/n$. For $0 \leq m \leq \binom{n}{2}$, we will also denote by $ \mathrm{G}(n,m)$ a uniform random graph with vertex set $\{1, \dots, n\}$ and exactly $m$ edges.

\begin{theorem}[]\label{thm:main}
	For $i \geq 2$, let $D_i(n)$ be the number of vertices of degree $i$ in the Karp--Sipser core of $ \mathrm{G}[n, \mathrm{e}/n]$. Then we have the joint convergence in distribution
	\begin{equation}\label{eqn:main_degree_convergence}
	\left(\begin{array}{c}n^{-3/5} \cdot D_{2}(n) \\  n^{-2/5} \cdot D_{3}(n) \\ n^{-1/5} \cdot D_4(n) \\ D_5(n) \end{array}\right)  \xrightarrow[n\to\infty]{(d)} \left( \begin{array}{c} \frac{2^{9/5}3^{4/5}}{\mathrm{e}^{3/5}} \vartheta^{-2} \\  
	\frac{2^{16/5}3^{1/5}}{\mathrm{e}^{2/5}} \vartheta^{-3} \\ \frac{2^{13/5}3^{3/5}}{\mathrm{e}^{1/5}} \vartheta^{-4} \\ \mathrm{Poi} \left( \frac{48}{5} \vartheta^{-5} \right) \end{array}\right),
	\end{equation}
	where $\vartheta=\inf\{t \geq 0 : W_t=t^{-2}\}$ for a standard Brownian motion $W$ started from $0$, and $\mathrm{Poi}(\lambda)$ stands for a Poisson variable with parameter $\lambda$. Moreover, with high probability, we have $D_i(n)=0$ for all $i \geq 6$. Finally, conditionally on $\left( D_i(n) \right)_{i \geq 2}$, the Karp--Sipser core is a configuration model\footnote{We refer to Section~\ref{sec:model} for a precise definition of the configuration model.} conditioned to be simple.
	Moreover, the same result is true if we replace $\mathrm{G}[n, \mathrm{e}/n]$ by $ \mathrm{G}(n,m_n)$ where $(m_n)_{ n \geq 1}$ is a sequence satisfying $m_n=\frac{\mathrm{e}}{2} n +O \left( \sqrt{n}\right)$.
\end{theorem}

\paragraph{Ideas of the proof and structure of the paper.}  To analyze the Karp--Sipser core of this graph, the global strategy is inspired by \cite{BCC22}. The first crucial observation is that the Karp--Sipser algorithm is Abelian, in the sense that the Karp--Sipser core does not depend on the order in which the vertices are removed. In particular, we can remove the leaves (and their neighours) one by one. In this setting, the execution of the algorithm can be described by an $\N^3$-valued Markov chain $(X_k, V_k, M_k)$ already introduced by Karp and Sipser~\cite{karp1981maximum}. The three coordinates of this Markov chain correspond respectively to the number $X_k$ of vertices of degree $1$, the number $V_k$ of vertices of degree at least $2$ and the total number $M_k$ of edges after $k$ steps of the algorithm. Understanding the size of the Karp--Sipser core is then equivalent to understanding the value of this chain at extinction, i.e. when $X_k$ hits $0$ for the first time.
Rather than the total number of edges, it will be convenient to consider instead the \emph{surplus} of the graph defined by $S_k=2 M_k-X_k-2V_k$. This is the number of additional half-edges of the graph compared to a graph where vertices would only have degree $1$ or $2$. In order to have exactly a Markov  chain, it will also be more convenient to consider a multigraph version of the Erd\H{o}s--R\'enyi random graph (see Section \ref{sec:model} for precise definitions, and Theorem~\ref{thm:multi} for the analog of Theorem~\ref{thm:main} on the multigraph model).

A standard tool to estimate $\R^d$-valued Markov chains is the \emph{differential equation method}. In our case, this method gives the convergence of the renormalized process $\left( (X_{ \lfloor  t n \rfloor}, V_{ \lfloor  t n \rfloor}, S_{ \lfloor  t n \rfloor} )/n \right)_{t \geq 0}$ towards a deterministic limit which we call its \emph{fluid limit}. This limit is characterized by a system of differential equations (see Section \ref{sec:fluid_limit}). However, since the size of the core is $o(n)$, this is not accurate enough close to the extinction time, so we will need to make this fluid limit approximation more quantitative. As in \cite{BCC22}, the first step is to show that $ \varepsilon n$ steps before the extinction, the three coordinates $X_k$, $V_k$ and $S_k$ are typically of order respectively $ \varepsilon^2 n$, $ \varepsilon n$ and $ \varepsilon^{3/2} n $. A difficulty here is that, while the transitions of the Markov chain in~\cite{BCC22} are completely explicit, here they depend on an implicit parameter $z$ (defined so that the degrees of non-leaf vertices are well-approximated by Poisson variables with parameter $z$, conditioned to be at least $2$). Together with the fact that vertex degrees are not bounded anymore, this will make the computations more technical.
Our estimates on the transitions of the Markov chain will mostly rely on results from~\cite{AFP97}, which in particular already computed its fluid limit. However, some of the error terms from~\cite{AFP97} are a bit too large for our purpose, so we will need to reprove some estimates in Sections~\ref{subsec:asymptotic_drift_v1} and~\ref{subsec:variance_v1}. In Section~\ref{sec:fluid_limit}, we will then gather some estimates on the continuous functions appearing in the fluid limit close to their extinction.

After that, as in~\cite{BCC22}, we will control precisely the drift and variance of the Markov chain in the neighbourhood of the trajectory given by the fluid limit approximation (Section~\ref{sec:drift_variance}). These estimates will allow us to prove that the fluctuations on the number of leaves $\eps n$ steps before extinction are typically of order $ \varepsilon^{3/4}  \sqrt{n}$. As in~\cite{BCC22}, we will argue that the extinction time arises when the fluctuations have the same order of magnitude as the expectation $\approx \eps^2 n$, that is when $ \varepsilon \approx n^{-2/5}$. In particular, the size of the Karp--Sipser core is of order $ \varepsilon n \approx n^{3/5}$. Moreover, the vertex degrees in the Karp--Sipser core will be well approximated by Poisson variables with parameter $z \approx \sqrt{\eps}$ conditioned to be at least $2$, which is where the various orders of magnitude of Theorem~\ref{thm:main} come from. Checking that these approximations remain true even when $ \varepsilon \approx n^{-2/5}$ requires a careful control of the Markov chain along different scales (Section \ref{sec:good_region}). This step does not introduce new ideas compared to~\cite{BCC22}.

The last step, which was not needed in~\cite{BCC22}, consists of coming back from the multigraph model of Theorem~\ref{thm:multi} to the simple graph models $  \mathrm{G}[n, \mathrm{e}/n]$ and $ \mathrm{G}(n,m_n)$ of Theorem~\ref{thm:main}. Since $ \mathrm{G}(n,m)$ is a random multigraph conditioned to be simple, we will need to carefully study the impact of this conditioning on the Karp--Sipser algorithm (Section~\ref{sec:back}). We highlight that even though the use of a multigraph model is a classical idea, the most straightforward approaches to deduce Theorem~\ref{thm:main} from Theorem~\ref{thm:multi} fail because the Karp--Sipser core is a very sensitive function of the graph. For this reason, we will actually need a stronger version of Theorem~\ref{thm:multi} (Theorem~\ref{thm:bis}), which roughly says that conditioning on the value of the Markov chain $\eta n$ steps before extinction for some fixed $\eta>0$ does not affect the law of the core.


\paragraph{Comparison with previous works.}
Analyzing the Markov chain $(X_k ,V_k, S_k)$ to study the Karp--Sipser algorithm on Erd\H{o}s--R\'enyi random graphs is far from being a new idea. However, in the previous works~\cite{karp1981maximum, AFP97, kreacic2017some}, it was usually sufficient to estimate this Markov chain up to error terms of order $o(n)$ (fluid limit via the differential equation method) or $o(\sqrt{n})$ (Gaussian fluctuations via stochastic differential equations). Here, we will need to control $X_k$ at the order $n^{1/5}$ in the end of the process. A key feature that we need to understand here is the ``self-correcting" effect of the Markov chain: while the fluctuations on the number $X_k$ of leaves are of order $\sqrt{n}$ during most of the algorithm, they become much smaller when we get close to extinction.

Moreover, many of these works were limited to the subcritical or supercritical regimes, where the limiting drift of the Markov chain behaves pretty nicely around $0$. In the critical regime, the fluid limits becomes more degenerate close to its extinction, which explains that the chain is of order $\left( \eps^2 n, \eps n, \eps^{3/2} n \right)$ before extinction. The results of~\cite[Sections 4, 5]{glasgow2024central} can be interpretated as estimates on this Markov chain at time $\eps n$ before extinction in the regime where $n \to +\infty$ first, and then $\eps \to 0$. On the other hand, the estimates that we obtain in the present work (Proposition~\ref{prop:roughbounds}) cover the whole range from $\eps$ of order $1$ to $n^{-2/5}$. In particular, it seems likely to us that the techniques developped here would provide a more precise estimation of the expected matching number in the central limit theorem of~\cite{glasgow2024central}.

\tableofcontents

\section{A three-dimensional Markov chain}\label{sec:Markov_chain}

\subsection{Definitions}\label{sec:model}

As mentioned above, the execution of the Karp--Sipser algorithm on $ \mathrm{G}[n, \mathrm{e}/n]$ can be described by a triple with integer values representing the number of leaves, the number of vertices of degree at least $2$ and the number of ``surplus" half-edges. However, the multigraph version of the model enjoys a nicer Markov property compared to simple graphs. This motivates the following definitions.

\paragraph{Random (multi)graphs with a fixed number of edges.} For $0 \leq m \leq \binom{n}{2}$, we define $ \mathrm{G}(n,m)$ (resp. $ \mathbb{G}(n,m)$) as the uniform graph (resp. multigraph) on $n$ vertices with $m$ edges. More precisely, we use labels to define this graph to avoid symmetry issues: a \emph{labelled multigraph} with $m$ edges on the set of vertices $\{1,\dots,n\}$ is defined by its edge-sequence $\left( a_i, b_i \right)_{1 \leq i \leq m}$, where $a_i, b_i \in \{1,\dots,n\}$ for all $1 \leq i  \leq m$. Note that it is possible that $a_i=b_i$, in which case the edge $i$ is a loop. Then the multigraph $ \mathbb{G}(n,m)$ is obtained by taking a uniform labelled multigraph with $m$ edges on the set of vertices $\{1,\dots,n\}$ and forgetting its edge-labels, and the graph $\mathrm{G}(n,m)$ has the law of $ \mathbb{G}(n,m)$ conditioned to be simple (i.e. without loops and multiple edges). 
We will first study the Karp--Sipser core of $\mathbb{G}(n,m)$ and then deduce our Theorem~\ref{thm:main}. For $d_1+\dots+d_n$ even, a \emph{configuration model} with vertex degrees $d_1, \dots, d_n$ is the random multigraph obtained by drawing $d_i$ half-edges around the vertex $i$ for all $1 \leq i \leq n$ and pairing those half-edges uniformly at random.

\begin{theorem}[]\label{thm:multi}
Let $(m_n)_{n \geq 1}$ be a (deterministic) sequence of nonnegative integers such that $m_n=\frac{\mathrm{e}}{2}n+O \left( \sqrt{n} \right)$.
For $i \geq 2$, let $\D_i(n)$ be the number of vertices of degree $i$ in the Karp--Sipser core of $ \mathbb{G}(n, m_n)$. Then we have the joint convergence in distribution
	\[
	\left(\begin{array}{c}n^{-3/5} \cdot \D_{2}(n) \\  n^{-2/5} \cdot \D_{3}(n) \\ n^{-1/5} \cdot \D_4(n) \\ \D_5(n) \end{array}\right)  \xrightarrow[n\to\infty]{(d)} \left( \begin{array}{c} \frac{2^{9/5}3^{4/5}}{\mathrm{e}^{3/5}} \vartheta^{-2} \\  
	\frac{2^{16/5}3^{1/5}}{\mathrm{e}^{2/5}} \vartheta^{-3} \\ \frac{2^{13/5}3^{3/5}}{\mathrm{e}^{1/5}} \vartheta^{-4} \\ \mathrm{Poi} \left( \frac{48}{5} \vartheta^{-5} \right) \end{array}\right),
	\]
	where $\vartheta=\inf\{t \geq 0 : W_t=t^{-2}\}$ for a standard Brownian motion $W$ started from $0$. Moreover, with high probability, we have $\D_i(n)=0$ for all $i \geq 6$. Finally, conditionally on $\left( \D_i(n) \right)_{i \geq 2}$, the Karp--Sipser core is a configuration model.
\end{theorem}

\paragraph{Fluctuations of the initial condition.}
For the rest of the paper (until Section~\ref{sec:back}), we fix a deterministic sequence $(m_n)_{n \geq 1}$ of nonnegative integers such that $m_n=\frac{\mathrm{e}}{2}n+O \left( \sqrt{n} \right)$.
Two important quantities in our analysis are the number of leaves and of vertices of degree at least $2$ in $ \mathbb{G}(n, m_n)$, that we denote respectively by $X_0^n$ and $V_0^n$. We also define $S_0^n = 2 m_n - X_0^n - 2 V_0^n$, which can be seen as the number of ``additional half-edges" compared to the case when all vertices of degree at least $2$ have degree exactly $2$.
We first estimate the triplet $( X_0^n, V_0^n, S_0^n)$, which will be the initial condition of our Markov chain. Heuristically, the vertex degrees in $\mathbb{G}(n, m_n)$ can be approximated by i.i.d. Poisson variables with parameter $\mathrm{e}$. Thus, at first order, there are $ \mathrm{e}^{1 - \mathrm{e}} n $ leaves and $(1 - \mathrm{e}^{- \mathrm{e}} - \mathrm{e}^{1 - \mathrm{e}}) n $ vertices of degree at least 2. The following result is close to~\cite[Lemma 4.3]{glasgow2024central}, except that it is for $m_n$ fixed instead of random. 

\begin{lemma}\label{lem:convergence_initial_conditions}
	The following vector is tight as $n \to +\infty$:
	\[ \frac{1}{\sqrt{n}} \left( X_0^n-\e^{1-\e} n, \, V_0^n- \left( 1-\e^{-\e}-\e^{1-\e} \right) n, \, S_0^n-\left( \e+\e^{1-\e}+2\e^{-\e}-2 \right) n \right). \]
\end{lemma} 

\begin{proof}First, we note that the degree distribution of the $n$ vertices in $ \mathbb{G} (n, m_n)$ corresponds exactly to the random allocation problem studied by Janson in \cite[Example 3.1]{janson2007monotonicity}, where $m = 2m_n$ balls (or half-edges) are thrown uniformly at random in $n$ boxes (or vertices). In particular, as an application of~\cite[Corollary 2.5]{janson2007monotonicity}, the normalized fluctuations
\[ \frac{1}{\sqrt{n}} \left( (X_0^n, V_0^n, S_0^n) - \left( \frac{2 m_n}{n} \mathrm{e}^{- \frac{2 m_n}{n}}, 1- \left( 1 + \frac{2 m_n}{n} \right) \e^{- \frac{2 m_n}{n}}, \frac{2 m_n}{n} + 2 \mathrm{e}^{- \frac{2 m_n}{n}}+ \frac{2 m_n}{n} \mathrm{e}^{- \frac{2 m_n}{n}} - 2 \right) n \right) \]
converge in distribution towards a Gaussian vector. The covariance matrix of the limit is explicit but we do not need it here. Moreover, since $m_n = \frac{\mathrm{e}n}{2} + O \left( \sqrt{n} \right)$, we have
\begin{multline*}
	\left( \frac{2 m_n}{n} \mathrm{e}^{- \frac{2 m_n}{n}}, 1- \left( 1 + \frac{2 m_n}{n} \right) \e^{- \frac{2 m_n}{n}}, \frac{2 m_n}{n} + 2 \mathrm{e}^{- \frac{2 m_n}{n}}+ \frac{2 m_n}{n} \mathrm{e}^{- \frac{2 m_n}{n}} - 2 \right) \\
	= \left( \mathrm{e}^{1- \mathrm{e}}, 1-\e^{-\e}-\e^{1-\e}, \e+\e^{1-\e}+2\e^{-\e}-2 \right) + O \left( \frac{1}{\sqrt{n}} \right),
\end{multline*}
which concludes the proof of the Lemma.
\end{proof}

\paragraph{Multigraphs with a fixed number of leaves, non-leaf vertices and edges.} An important distribution for us will be uniform multigraphs with given numbers of leaves, of vertices of degree at least $2$ and of edges.
For $x,v,s \geq 0$ such that $x + 2v+s$ is even, we consider a uniform random variable among all the labelled multigraphs on the vertex set $\{ 1, \dots, x+v \}$ with $\frac{x+2v+s}{2}$ edges and such that the vertices $1, \dots, x$ have degree $1$ and the vertices $x+1, \dots, x+v$ have degree at least $2$. We denote by $\GG(x,v,s)$ the random multigraph obtained by forgetting the edge-labels of this labelled multigraph. We call $s$ the \emph{surplus} of the multigraph, as it is the additional number of half-edges compared to the case where all vertices have degree either $1$ or $2$. 

\begin{lemma}\label{lem:config_model}
	Let $x,v,s \geq 0$ with $x + 2v+s$ even. Conditionally on the family of vertex degrees $(\deg(j))_{1 \leq j \leq x+v}$, the multigraph $\GG(x,v,s)$ is a configuration model. 
\end{lemma}
\begin{proof}
	The labelled multigraph $\GG(x,v,s)$ is uniform among all the labelled multigraphs with $x$~leaves, with $v$ vertices of degree at least $2$ and with surplus $s$, so it is still uniform if we condition on all the vertex degrees.
\end{proof}
In particular, the two random multigraphs $ \mathbb{G}(n, m_n)$ and $ \mathbb{G}(X_0^n,V_0^n,S_0^n)$ have the same law.

\paragraph{One step of the Karp--Sipser algorithm.} Equipped with this definition, we now want to apply the Karp--Sipser algorithm on the random multigraphs $\GG(x,v,s)$. Given a multigraph $G$ which is a possible realization of $ \mathbb{G} (x,v,s)$ for some $x \geq 1$ and $v,s \geq 0$, we denote by $G^-$ the multigraph obtained from $G$ by the following sequence of operations (see Figure \ref{fig:example}):
\begin{enumerate}
	\item remove the leaf with label $1$ and its unique neighbour,
	\item remove all the edges incident to that neighbour,
	\item remove the isolated vertices that might appear in the process,
	\item relabel the leaves in increasing order followed by the vertices of degree at least $2$ in increasing order. 
\end{enumerate}
\begin{figure}[!h]
 \begin{center}
 \includegraphics[width=16cm]{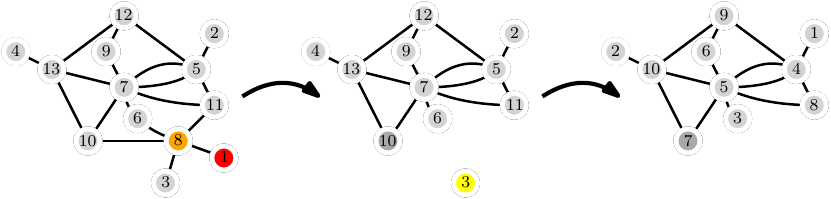}
 \caption{\label{fig:example}Illustration of one step of the Karp--Sipser algorithm. We first remove the leaf with label $1$ (in red) and its unique neighbour (in orange) and all edges incident to them. Then we remove the isolated vertices (in yellow) and relabel the leaves in increasing order followed by the vertices of degree at least $2$ in increasing order.}
 \end{center}
 \end{figure}

Our starting point is the following property, already observed in~\cite[Lemma 2]{AFP97}.

\begin{lemma}\label{lem:onestep}
	If $G$ follows the law of $\mathbb{G}(x,v,s)$, then conditionally on the number $X^-$ of leaves, the number $V^-$ of non-leaf vertices and the surplus $S^-$ of $G^-$, the graph $G^-$ follows the law of $ \mathbb{G}(X^-, V^-, S^-)$.
\end{lemma}
\paragraph{The Markov chain.}As a result, the Karp--Sipser algorithm yields a natural three-dimensional random process defined as follows. Let us start from  $ G_0^n := \mathbb{G}( X_0^n, V_0^n, S_0^n)$, where we recall that $( X_0^n, V_0^n, S_0^n)$ are respectively the number of leaves, vertices of degree at least $2$ and surplus of $ \mathbb{G}(n , m_n)$. 
For all $k \geq 0$, if $G_k^n$ has at least one leaf, let $G_{k+1}^n=(G_k^n)^-$. If $G_k^n$ has no leaf, we just set $G_{k+1}^n=G^n_k$, and we denote by $\theta^n$ the first time at which this is the case. Then $G_{\theta^n}^n$ is the Karp--Sipser core of $G_0^n$ (up to a relabelling of its vertices).
For all $k \geq 0$, we denote by $X_k^n$ (resp.\ $V_k^n$, $M_k^n$) the number of leaves (resp.\ non-leaf vertices, half-edges) of $G_k^n$, and by $S_k^n=2M_k^n-X_k^n-2V_k^n$ its surplus. We will also denote by $\left( \mathcal{F}_k^n \right)_{k \geq 0}$ the associated filtration, i.e. $\mathcal{F}_k^n$ is the $\sigma$-algebra generated by $\left( X_i^n, V_i^n, S_i^n \right)_{0 \leq i \leq k}$. As observed in~\cite[Assertion (1), p266]{karp1981maximum} or~\cite[Lemma 3]{AFP97}, this defines a Markov chain.

\begin{lemma}\label{lem:markov}
The process $(X_k^n, V_k^n, S_k^n)_{0 \leq k \leq \theta^n}$ is a Markov chain. Moreover, for any stopping time $T$, conditionally on $\Ff_T^n$, the graph $G_T^n$ has the law of $ \mathbb{G}(X_T^n, V_T^n,S_T^n)$.
\end{lemma}
The proof of the phase transition by Karp and Sipser~\cite{karp1981maximum} relies on this Markov chain. Using the \emph{differential equation method}~\cite{wormald1995differential} for the first time in the context of random graphs, they showed that the renormalized version of this Markov chain is  well approximated by a differential equation on $ \mathbb{R}^3$ for which they find an explicit solution, see Section~\ref{sec:fluid_limit} below. 

To simplify notation, the $n$ in the exponent will be implicit for the rest of the paper when there is no ambiguity, even if we will often look at asymptotics as $n \to +\infty$.

\subsection{The asymptotic drift of the Markov chain}\label{subsec:asymptotic_drift_v1}

The next step is to estimate the local drift of the Markov chain $(X,V,S)$ (Proposition~\ref{prop:drift_estimates_AFP} below). Most of the estimates that we need were already proved in~\cite[Lemma 6]{AFP97}. Unfortunately, in one of those estimates, the error term is too large for our purpose, which is why we will need to redo some of the proofs. As in~\cite{AFP97}, the first step is to understand the distribution of vertex degrees in the random graph $\mathbb{G}(x,v,s)$. We will see that the degrees of non-leaf vertices are well approximated by i.i.d. Poisson variables conditioned to be at least $2$. We will now define the parameter $z$ of these Poisson variables: for all $v,s>0$, we denote by $z(v,s)>0$ the unique solution to the equation
\begin{equation}\label{eqn:definition_z}
	\frac{z(\mathrm{e}^z-1)}{\mathrm{e}^z-z-1}=2+\frac{s}{v}.
\end{equation}
We also set the convention $z(v,0)=0$ for $v>0$, and we will write $Z_k=z(V_k,S_k)$. Note that in~\cite{AFP97}, the right-hand side is written in a different form, as it is expressed in terms of $x, v$ and the number of edges. Since the number of edges is equal to $\frac{1}{2}(x+2v+s)$, the adaptation is immediate.
It will be important for us to understand the behaviour of $z$ when $s$ and $v$ go to $0$. Although it is not a smooth function of $(v,s)$ as $v \to 0$, we can still write it as a smooth function of the ratio $\frac{s}{v}$.

\begin{lemma}\label{lem:z_smooth}
	There is a smooth function $\widetilde{z} : [0,+\infty) \to [0,+\infty)$ such that for all $v>0$ and $s \geq 0$, we have $z(v,s)=\widetilde{z} \left( \frac{s}{v} \right)$. Moreover, we have
	\begin{equation}\label{eqn:DL_z}
		z(v,s)=3\frac{s}{v} + O \left( \left( \frac{s}{v} \right)^2 \right).
	\end{equation}
\end{lemma}

In particular, we will see later that $\eps n$ steps before extinction, the quantities $S_k$ and $V_k$ are respectively of order $\eps^{3/2} n$ and $\eps n$, so $Z_k$ will be of order $\eps^{1/2}$.

\begin{proof}
	We first verify that the derivative of the left-hand side of~\eqref{eqn:definition_z} with respect to $z$ is positive. Specifically, we have
	\begin{equation}\label{eqn:derivative_z}
	\frac{ \mathrm{d}}{ \mathrm{d}z} \frac{z(\mathrm{e}^z-1)}{\mathrm{e}^z-z-1} = \frac{\mathrm{e}^{2z}-2\mathrm{e}^z+1-z^2 \mathrm{e}^z}{(\mathrm{e}^z-z-1)^2} = \frac{\mathrm{e}^z \left( 2 \cosh(z) -2-z^2 \right)}{(\mathrm{e}^z-z-1)^2}.
	\end{equation}
	The numerator is positive for $z>0$ by expanding $\cosh(z)$ into a power series. Note that this proves that the solution to~\eqref{eqn:definition_z} is indeed unique. Moreover~\eqref{eqn:definition_z} defines $\widetilde{z} \left( \frac{s}{v} \right)$ for $\frac{s}{v} \ne 0$ and can be extended to a continuous function by setting $\widetilde{z}(0)=0$. Since~\eqref{eqn:derivative_z} does not vanish, the function $\widetilde{z}$ is smooth on $(0,+\infty)$. Therefore, to conclude, all we need to check is that~\eqref{eqn:derivative_z} does not vanish as $z\to0$. But in this regime, the numerator of~\eqref{eqn:derivative_z} is equivalent to $\frac{1}{12}z^4$ and its denominator to $\frac{z^4}{4}$, so $\frac{d}{dz} \frac{z(\mathrm{e}^z-1)}{\mathrm{e}^z-z-1}>0$ on $[0,+\infty)$.
	
	Finally, the second part~\eqref{eqn:DL_z} immediately follows from the fact that the left-hand side of~\eqref{eqn:definition_z} is $2+\frac{z}{3}+O(z^2)$ as $z \to 0$ and is of order $z$ as $z \to +\infty$. 
\end{proof}

\paragraph{Poisson approximation for vertex degrees.} We now show as in~\cite[Lemmas 4, 5]{AFP97} that vertex degrees are well approximated by i.i.d. Poisson variables conditioned to be at least $2$. As it will appear a lot, we define $f(z)=\mathrm{e}^z-z-1$.

\begin{lemma}\label{lem:degrees_are_Poisson}
	Let $C>0$ be a constant and assume that $V_k Z_k \geq \log^3 n$ and $Z_k \leq C$. In what follows, the constants involved in the $O$ notation will only depend on $C$.
	\begin{enumerate}
		\item For all $X_k+1 \leq j \leq X_k+V_k$ and for all $3 \leq i \leq \log n$, we have
			\begin{equation}\label{eqn:degree_distribution_moderate}
			\P \left( \deg_{G_k}(j)=i | \Ff_k \right) = \frac{Z_k^i}{i! f(Z_k)} \left( 1+O \left( \frac{i^2}{V_k Z_k} \right) \right).
			\end{equation}
			Moreover, if $i=2$, the $O \left( \frac{i^2}{V_k Z_k} \right)$ error term can be replaced by $O \left( \frac{1}{V_k} \right)$.
			Also, for all $i \geq 2$, we have
			\begin{equation}\label{eqn:degree_distribution_high}
				\P \left( \deg_{G_k}(j)=i | \Ff_k \right) = O \left( \sqrt{V_k Z_k} \frac{Z_k^i}{i!f(Z_k)} \right).
			\end{equation}
			In particular, we have
			\begin{equation}\label{eqn:degree_distribution_light_tail}
				\P \left( \deg_{G_k}(j) \geq \log n | \Ff_k \right) = O \left( n^{-10} \right).
			\end{equation}
			\item For all $X_k+1 \leq j_1<j_2 \leq X_k+V_k$ and $2 \leq i_1, i_2 \leq \log n$, we have
			\begin{equation}\label{eqn:degree_distribution_two_pt_function}
				\P \left( \deg_{G_k}(j_1)=i_1 \mbox{ and } \deg_{G_k}(j_2)=i_2 | \Ff_k \right) = \frac{Z_k^{i_1}}{i_1 ! f(Z_k)} \frac{Z_k^{i_2}}{i_2 ! f(Z_k)} \left( 1+O \left( \frac{\log^2 n}{V_k Z_k} \right) \right).
			\end{equation}
			Moreover, if $i_1=i_2=2$, the error term can be replaced by $O \left( \frac{\log^2 n}{V_k} \right)$.
			\item For all $X_k+1 \leq j_1<j_2<j_3 \leq X_k+V_k$ and $2 \leq i_1, i_2, i_3 \leq \log n$, we have
			\begin{equation}\label{eqn:degree_distribution_three_pt_function}
			\P \left( \forall a \in \{1,2,3\}, \deg_{G_k}(j_a)=i_a | \Ff_k \right) = \left( \prod_{a=1}^3 \frac{Z_k^{i_a}}{i_a ! f(Z_k)} \right) \times \left( 1+O \left( \frac{\log^2 n}{V_k Z_k} \right) \right).
			\end{equation}
			\item Let $m \geq 1$. Then for all $X_k+1 \leq j_1<\dots<j_m \leq X_k+V_k$, we have
			\begin{equation}\label{eqn:m_pt_function_degree_five}
				\P \left( \deg_{G_k}(j_1)=\dots=\deg_{G_k}(j_m)=5 | \Ff_k \right) = (1+o(1)) \left( \frac{Z_k^5}{5! f(Z_k)} \right)^m,
			\end{equation}
			where the $o$ only depends on $m, C$.
		\end{enumerate}
\end{lemma}

\begin{proof}
	\begin{enumerate}
		\item This item is exactly the same as~\cite[Lemma 5]{AFP97}, except for~\eqref{eqn:degree_distribution_light_tail} and for the improved error term $O \left( \frac{1}{V_k} \right)$ in~\eqref{eqn:degree_distribution_moderate} for $i=2$. We first deduce~\eqref{eqn:degree_distribution_light_tail} from~\eqref{eqn:degree_distribution_high}. Using~\eqref{eqn:degree_distribution_high}, we can write
		\[ \P \left( \deg_{G_k} (j) \geq \log n | \Ff_k \right) = O \left( \sqrt{V_k Z_k} \sum_{i > \log n} \frac{Z_k^i}{i! f(Z_k)} \right).  \]
		Using $f(z) \geq c \times z^2$ on $[0,C]$ for some constant $c>0$ and $i! \geq i^i \mathrm{e}^{-i}$, we deduce
		\begin{align*}
		\P \left( \deg_{G_k} (j) \geq \log n | \Ff_k \right) &= O \left( n^{1/2} \sum_{i > \log n} \frac{Z_k^{i-2} \mathrm{e}^i}{i^i} \right)\\
		&= O \left( n^{1/2} \sum_{i > \log n} \left( \frac{ \mathrm{e}C}{\log n} \right)^i \right)\\
		&= O \left( n^{1/2} \frac{( \mathrm{e}C)^{\log n}}{(\log n)^{\log n}} \right).
		\end{align*}
		Since the numerator is polynomial in $n$ and the denominator is larger than any 	polynomial, this is $O(n^{-10})$.
	
		We now move on to the case $i=2$ of~\eqref{eqn:degree_distribution_moderate}. For this, using~\eqref{eqn:degree_distribution_moderate} for $i \geq 3$ and~\eqref{eqn:degree_distribution_light_tail}, we can write
		\begin{align}
			\P \left( \deg_{G_k}(j)=2 | \Ff_k \right) &= 1-\sum_{i \geq 3} \P \left( 	\deg_{G_k}(j)=i | \Ff_k \right)\nonumber \\
			&= 1-\sum_{i=3}^{\log n} \frac{Z_k^i}{i! f(Z_k)}+O \left( \frac{1}{V_k Z_k f(Z_k)}\sum_{i=3}^{\log n} \frac{i^2}{i!} Z_k^{i} \right) + O \left( n^{-10} \right). \label{eqn:proba_deg_two_decomposition}
		\end{align}
		Since $0<Z_k \leq C$, the quantity $\frac{Z_k^2}{f(Z_k)}$ is bounded away from $0$ and $+\infty$, so what we want to show is that~\eqref{eqn:proba_deg_two_decomposition} is $\frac{Z_k^2}{2f(Z_k)}+O \left( \frac{1}{V_k} \right)$. First, by definition of $f$ we have $\sum_{i \geq 2} \frac{z^i}{i!f(z)}=1$ for all $z$, so
		\[ 1-\sum_{i=3}^{\log n} \frac{Z_k^i}{i! f(Z_k)}=\frac{Z_k^2}{2f(Z_k)}+O \left( \sum_{i>\log n} \frac{Z_k^i}{i!f(Z_k)} \right)= \frac{Z_k^2}{2f(Z_k)}+O \left( n^{-10} \right)\]
		by the same computation as in the proof of~\eqref{eqn:degree_distribution_light_tail}. In particular, the error term is $O \left( \frac{1}{V_k} \right)$. We now have to handle the first error term of~\eqref{eqn:proba_deg_two_decomposition}. We rewrite it as
		\[ O \left( \frac{Z_k^2}{V_k f(Z_k)} \sum_{i=3}^{\log n} \frac{i^2}{i!}Z_k^{i-3} \right) =O \left( \frac{1}{V_k} \sum_{i \geq 3} \frac{i^2}{i!}C^{i-3} \right) = O \left( \frac{1}{V_k} \right). \]
		The second error term is $O \left( n^{-10} \right)$ by~\eqref{eqn:degree_distribution_light_tail}, which concludes the proof.
		\item This is proved in~\cite[Lemma 5]{AFP97}, except for two small modifications:
		\begin{itemize}
			\item in~\cite[Lemma 5]{AFP97}, the corresponding estimate was only stated for $i_1, i_2 \leq \log V_k$ and not for $\log n$, but the error term was $\frac{\log^2 V_k}{V_k Z_k}$ and not $\frac{\log^2 n}{V_k Z_k}$. However, the proof remains exactly the same, relying on~\cite[Equations (5), (6)]{AFP97};
			\item again, we claim an improved error term in the case $i_1=i_2=2$. For this improvement, we use the exact same argument as for $i=2$ in Item 1, writing
			\begin{multline*}
			\P \left( \deg_{G_k}(j_1)=2 \mbox{ and } \deg_{G_k}(j_2)=2 | \Ff_k \right)\\ = \P \left( \deg_{G_k}(j_1)=2 | \Ff_k \right) -\sum_{i_2 \geq 3} \P \left( \deg_{G_k}(j_1)=2 \mbox{ and } \deg_{G_k}(j_2)=i_2 | \Ff_k \right).
			\end{multline*}
		\end{itemize}
		\item The proof of the three-point estimate~\eqref{eqn:degree_distribution_three_pt_function} is exactly the same as the proof of the two-point estimate in~\cite[Lemma 5]{AFP97}, with the same remark as above on replacing $\log V_k$ by $\log n$.
		\item Again, the argument is the same as for the two-point estimate. Note that the error $\frac{\log^2 n}{V_k Z_k}$ is indeed $o(1)$ (this is why we assumed $V_k Z_k \geq \log^3 n$, which is slightly stronger than in~\cite{AFP97}) and that the error $o(1)$ can depend on $m$, so we do not need a uniform statement in $m$.
	\end{enumerate}
\end{proof}

\paragraph{Asymptotic drift.}
We now define the functions that will describe the drift of our Markov chain $(X,V,S)$ in the fluid limit. In the equations below, we will just write $z$ instead of the $z(v,s)$ given by~\eqref{eqn:definition_z}. For $x,v,s>0$, we write
	\begin{align}
		\Phi_A(x,v,s) &= -1-\frac{x}{x+2v+s}+\frac{v^2 z^4 \mathrm{e}^z}{(x+2v+s)^2 f(z)^2}-\frac{xvz^2 \mathrm{e}^z}{(x+2v+s)^2f(z)}, \label{eqn:defn_Phi_A} \\
		\Phi_B(x,v,s) &= -1+\frac{x}{x+2v+s}-\frac{v^2 z^4 \mathrm{e}^z}{(x+2v+s)^2 f(z)^2}, \label{eqn:defn_Phi_B} \\
		\Phi_C(x,v,s) &= 1-\frac{x}{x+2v+s}-2\frac{v z^2 \mathrm{e}^z}{(x+2v+s)f(z)}+\frac{v^2 z^4 \mathrm{e}^z}{(x+2v+s)^2 f(z)^2}+\frac{x v z^2 \mathrm{e}^z}{(x+2v+s)^2 f(z)}. \label{eqn:defn_Phi_C}
	\end{align}
	It was proved in~\cite[Lemma 6]{AFP97} that the fluid limit of the Markov chain $(X,V,S)$ satisfies a differential equation given by these three functions. We write $\Delta X_k$ for $X_{k+1}-X_k$, and similarly for the other processes. As we need slightly better error terms than those of~\cite{AFP97}, we will reprove some of those estimates.

\begin{proposition}\label{prop:drift_estimates_AFP}
	Let $C>0$. If $V_k Z_k \geq \log^3 n$ and $Z_k\leq C$ and $X_k>0$, then we have
	\begin{align}
		\E \left[ \Delta X_k | \Ff_k \right] &= \Phi_A \left( X_k, V_k, S_k \right) + O \left( \frac{\log^2 n}{V_k} \right), \label{eqn:drift_v1_X}\\
		\E \left[ \Delta V_k | \Ff_k \right] &= \Phi_B \left( X_k, V_k, S_k \right) + O \left( \frac{\log^2 n}{V_k Z_k} \right), \label{eqn:drift_v1_V} \\
		\E \left[ \Delta S_k | \Ff_k \right] &= \Phi_C \left( X_k, V_k, S_k \right) + O \left( \frac{\log^2 n}{V_k Z_k} \right), \label{eqn:drift_v1_H}
	\end{align}
	where the constants implied by the $O$ notation only depend on $C$.
\end{proposition}

\begin{proof}
	To make the following computations more compact, we introduce the number $H_k:=2V_k+X_k+S_k = 2M_k$ of half-edges of the graph $G_k$. Note that $H_k$ corresponds to $2m$ in the notations of~\cite[Lemma 6]{AFP97}, so all the error terms $O \left( \frac{1}{m} \right)$ of~\cite{AFP97} can be replaced by $O \left( \frac{1}{H_k} \right)$, which is also $O \left( \frac{1}{V_k} \right)$.
	
	Equations~\eqref{eqn:drift_v1_V} and~\eqref{eqn:drift_v1_H} are already proved in~\cite[Lemma 6]{AFP97} (with the natural substitution of the number $2m$ of half-edges by $x+2v+s$), so there is nothing to do. On the other hand~\eqref{eqn:drift_v1_X} is only proved in~\cite{AFP97} with an error term $O \left( \frac{\log^2 V_k}{V_k Z_k} \right)$, which we need to improve. However, our argument will follow~\cite{AFP97} closely.
	
	Let $k \geq 0$ be such that $V_k$, $X_k$ and $Z_k$ satisfy the assumptions of the Proposition. We recall that to pass from $G_k$ to $G_{k+1}$, we start by removing the leaf $1$ of $G_k$, and  we denote by $y$ its unique neighbour in $G_k$. All of the vertex degrees in the proof below will be degrees in the graph $G_k$. We denote by $D_k$ the degree of $y$ and by $D_k^{(1)}$ (resp. $D_k^{(2)}$, $D_k^{(3)}$) the number of neighbours of degree $1$ (resp. $2$, at least $3$) of $y$ in $G_k$, with the leaf $1$ counted in $D_k^{(1)}$. Let also $\widetilde{D}_k$ be the total number of loops and multiple edges incident to $y$. It is proved in~\cite[Equation~(8)]{AFP97} that
	\begin{equation}\label{eqn:very_few_loops}
		\E \left[ \widetilde{D}_k | \Ff_k \right] = O \left( \frac{1}{H_k} \right) = O \left( \frac{1}{V_k} \right).
	\end{equation}
	Moreover, we can write
	\begin{equation}\label{eqn:decomposition_DeltaX}
		\Delta X_k = -D_k^{(1)}-\mathbbm{1}_{D_k=1}+D_k^{(2)}+O \left( \widetilde{D}_k \right).
	\end{equation}
	Indeed, the leaves that are removed are the neighbours of $y$, plus $y$ itself if it has degree $1$, whereas the leaves which are created are the neighbours of $y$ with degree $2$ (unless they are linked to $y$ by multiple edges, which is accounted for by the term $O \left( \widetilde{D}_k \right)$). We now estimate the expectations of these terms one by one. First, by Lemma~\ref{lem:config_model}, we have
	\begin{equation}\label{eqn:proba_y_is_leaf}
		\P \left( D_k=1 | \Ff_k \right) = \frac{X_k-1}{H_k-1}.
	\end{equation}
	Moreover, let $\mathcal{D}_k$ be the family of vertex degrees of $G_k$. By Lemma~\ref{lem:config_model} again, we can write
	\[ \E \left[ D_k^{(1)} | \mathcal{F}_k, \mathcal{D}_k \right] = 1+O(\widetilde{D}_k)+ \sum_{j;\deg(j) \geq 2} \frac{\deg(j)}{H_k-1} \cdot \left( \deg(j)-1 \right) \cdot \frac{X_k-1}{H_k-3},  \]
where the sum over $j$ corresponds to the possible values of $y$ and the error term accounts for the possibility of loops and multiple edges around $y$. We can now integrate over $\mathcal{D}_k$ and use~\eqref{eqn:very_few_loops} and Lemma~\ref{lem:degrees_are_Poisson}. By decomposing according to the possible values $d$ of $\deg(j)$, we get
	\begin{align}
		\E \left[D_k^{(1)} | \mathcal{F}_k  \right] &= 1+O \left( \frac{1}{V_k} \right) + \frac{X_k-1}{(H_k-1)(H_k-3)} \sum_{j=X_k+1}^{X_k+V_k} \E \left[ \deg(j) \left( \deg(j)-1 \right) | \Ff_k \right] \nonumber \\
		&= 1+O \left( \frac{1}{V_k} \right)+\frac{(X_k-1) V_k}{(H_k-1)(H_k-3)} \cdot \Bigg( \frac{2 Z_k^2}{2f(Z_k)} \left( 1+O \left( \frac{1}{V_k} \right) \right)	\nonumber \\
		&+ \sum_{d=3}^{\log n} \frac{d(d-1) Z_k^d}{d! f(Z_k)} \left( 1+O \left( \frac{\log^2 n}{V_k Z_k} \right) \right) + O \left( n^{-8} \right) \Bigg),\label{eqn:D1_decomposed_by_degree}
	\end{align}
	where the last term is the contribution of $\deg(j) \geq \log n$ and comes from~\eqref{eqn:degree_distribution_light_tail}. The second error term is $O \left( \frac{X_k Z_k^2}{H_k^2 f(Z_k)} \right)$, which is $O \left( \frac{1}{V_k}\right)$ since $Z_k^2=O(f(Z_k))$ and $X_k, V_k \leq H_k$ (we are here slightly more accurate than~\cite{AFP97}). On the other hand, for $3 \leq d \leq \log n$, we have $\frac{Z_k^d}{f(Z_k)}=\frac{Z_k^2}{f(Z_k)} \cdot Z_k \cdot Z_k^{d-3}=O \left( C^{d-3} Z_k \right)$, where $C$ is an upper bound on $Z_k$ as in the statement of the Proposition. Hence, the corresponding error term is $O \left( \frac{C^d \log^2 n}{(d-2)! V_k} \right)$. Therefore, the sum of the error terms over $3 \leq d \leq \log n$ is $O \left( \frac{\log^2 n}{V_k} \right)$. Finally, the last error term is $O \left( \frac{\log^2 n}{V_k} \right)$. Summing the main terms as in~\cite[Equation~(17)]{AFP97}, we obtain
	\begin{equation}\label{eqn:expectation_neighbours_degree_one}
		\E \left[D_k^{(1)} | \mathcal{F}_k  \right] = 1+\frac{X_k V_kZ_k^2 \mathrm{e}^{Z_k}}{H_k^2 f(Z_k)} + O \left( \frac{\log^2 n}{V_k} \right).
	\end{equation}
	We now estimate the number of neighbours of $y$ with degree $2$ using similar ideas. We first write
	\[ \E \left[ D_k^{(2)} | \Ff_k, \mathcal{D}_k \right] = O \left( \widetilde{D}_k \right) + \sum_{j_1; \, \deg(j_1) \geq 2} \, \sum_{j_2; \, \deg(j_2)=2} \frac{\deg(j_1)}{H_k-1} \times (\deg(j_1)-1) \times \frac{2}{H_k-3}, \]
	where $j_1$ corresponds to possible values of $y$, and $j_2$ corresponds to possible values of neighbours of $y$ other than $1$. Again, we integrate over $\mathcal{D}_k$ and use~\eqref{eqn:very_few_loops} to bound the  error term:
	\[ \E \left[ D_k^{(2)} | \Ff_k \right] =O \left( \frac{1}{V_k} \right) + \frac{2}{(H_k-1)(H_k-3)} \sum_{j_1=X_k+1}^{X_k+V_k} \, \sum_{j_2=X_k+1}^{X_k+V_k} \E \left[ \deg(j_1) (\deg(j_1)-1) \mathbbm{1}_{\deg(j_2)=2} \right]. \]
	We then decompose each term according to the value of $\deg(j_1)$ as in~\eqref{eqn:D1_decomposed_by_degree} and apply the estimate~\eqref{eqn:degree_distribution_two_pt_function} from Lemma~\ref{lem:degrees_are_Poisson}. The error terms are the same as in the computation for $D_1^{(k)}$ so, using our improved error in~\eqref{eqn:degree_distribution_two_pt_function} for $i_1=i_2=2$, we find
	\begin{equation}\label{eqn:expectation_neighbours_degree_two}
		\E \left[ D_k^{(2)} | \Ff_k \right] = \frac{V_k^2 Z_k^4 \mathrm{e}^{Z_k}}{H_k^2 f(Z_k)^2} + O \left( \frac{\log^2 n}{V_k} \right),
	\end{equation}
	which again is a slight improvement compared to~\cite[Equation~(18)]{AFP97}. Finally, we obtain~\eqref{eqn:drift_v1_X} with the right error term by plugging~\eqref{eqn:very_few_loops},~\eqref{eqn:proba_y_is_leaf},~\eqref{eqn:expectation_neighbours_degree_one} and~\eqref{eqn:expectation_neighbours_degree_two} in~\eqref{eqn:decomposition_DeltaX}.
\end{proof}

\subsection{General variance estimates}\label{subsec:variance_v1}

In order to show that the Markov chain stays close to its fluid limit, we will also need to control its variance. As we will see later in Section~\ref{subsec:SDE}, it will be particularly important for us to have a precise estimate on $\var(\Delta X_k)$ in the end of the process. To fix ideas, we mention right now that $\eps n$ steps before the end of the process, the ratio $\frac{X_k}{V_k}$ will be of order $\eps$ and $Z_k$ of order $\eps^{1/2}$.

\begin{proposition}\label{prop_variance_estimates_v1}
	Let $C>0$. If $V_k Z_k \geq \log^3 n$ and $Z_k\leq C$ and $X_k>0$, then we have
	\begin{align}
		\var \left( \Delta V_k | \Ff_k \right) &= O(1), \label{eqn:variance_V}\\ 
		\var \left( \Delta S_k | \Ff_k \right) &= O(Z_k)+O \left( \frac{X_k}{V_k} \right), \label{eqn:variance_S} \\ 
		\var \left( \Delta X_k | \Ff_k \right) &= Z_k + O \left( Z_k^2 \right) + O \left( \frac{X_k}{V_k} \right) + O \left( \frac{\log^2 n}{V_k} \right), \label{eqn:variance_X}
	\end{align}
	where the constants implied by the $O$ notation only depend on $C$.
\end{proposition}

\begin{proof}
	We keep the notation of the proof of Proposition~\ref{prop:drift_estimates_AFP}. We start by proving~\eqref{eqn:variance_V}. Note that $0 \geq \Delta V_k \geq -D_{k}$, so it is enough to prove $\E \left[ D_k^2 | \Ff_k \right]=O(1)$. For this, we use the same ideas as in the proof of Proposition~\ref{prop:drift_estimates_AFP}. We first note that $D_k \leq n$ so by~\eqref{eqn:degree_distribution_light_tail}, the contribution of the event $\{ D_k>\log n\}$ is $O \left( \frac{1}{n} \right)$. Summing over all possible values of the neighbour $y$ of the removed leaf and over all possible values of $D_k$ between $3$ and $\log n$ and using Lemma~\ref{lem:degrees_are_Poisson}, we find
	\[ \E \left[ D_k^2 \mathbbm{1}_{D_k \geq 3} | X_k,V_k,S_k \right] = O \left( \frac{1}{n} \right) + O \left( \frac{V_k}{H_k-1} \sum_{d=2}^{\log n} d^3 \frac{Z_k^d}{d! f(Z_k)} \right)=O(1) \]
	using $V_k \leq H_k-1$ and $\frac{Z_k^d}{f(Z_k)}=O(Z_k^{d-2})=O(1)$.
	
	The argument for~\eqref{eqn:variance_S} is similar, but we also need to notice that if both $Z_k$ and $\frac{X_k}{V_k}$ are very small, then the most likely case is that $y$ has degree $2$ and its second neighbour as well, i.e. $D_k=2$ and $D_k^{(1)}=D_k^{(2)}=1$. In this case, we have $\Delta X_k=0$ but $\Delta V_k=-2$ and $\Delta H_k=-4$, so $\Delta S_k=0$. Therefore, we can write the bound \[\var \left( \Delta  S_k | \Ff_k \right) = O \left( \E \left[ D_k^2 \left( \mathbbm{1}_{D_k=1} + \mathbbm{1}_{D_k \geq 3} + \mathbbm{1}_{D_k=2, D_k^{(1)}=2} + \mathbbm{1}_{D_k=2, D_k^{(3)}=1} \right) \Big| \Ff_k \right] \right). \]
	We can then bound each of the four terms one by one. We find that the contributions of the first and third terms are $O \left( \frac{X_k}{V_k} \right)$, whereas the second and fourth are $O \left( Z_k \right)$. Let us illustrate this with the proof for the fourth term, as it is the most complicated one. We sum over all possible values of $y, w, d$, where $y$ is the unique neighbour of the removed leaf $1$ and $w$ is the second neighbour of $y$, and where $d \geq 3$ is the degree of $w$.
	\[ \E \left[ D_k^2 \mathbbm{1}_{D_k=2, D_k^{(3)}=1} | \Ff_k \right] = \sum_{d \geq 3} \sum_{y=X_k+1}^{X_k+V_k} \sum_{\substack{w=X_k+1\\ w \ne y}}^{X_k+V_k} 4 \P \left( \deg(y)=2 \mbox{ and } \deg(w)=d \right) \times \frac{2}{H_k-1} \times \frac{d}{H_k-3}. \]
	We now use Lemma~\ref{lem:degrees_are_Poisson}. Using~\eqref{eqn:degree_distribution_light_tail} for $d \geq \log n$ and~\eqref{eqn:degree_distribution_two_pt_function} for $3 \leq d \leq \log n$, we find
	\begin{align*}
		\E \left[ D_k^2 \mathbbm{1}_{D_k=2, D_k^{(3)}=1} | \Ff_k \right] &= O(n^{-8}) + O \left( V_k^2 \sum_{d=3}^{\log n} \frac{d}{H_k^2} \frac{Z_k^{d+2}}{d!f(Z_k)^2} \right)\\
		&= O \left( n^{-8} \right) + O \left( Z_k  \sum_{d\geq 3} \frac{d}{d!} C^{d-3}  \right)\\
		&= O \left( Z_k \right),
	\end{align*}
	where we have used that $f(Z_k)$ is comparable to $Z_k^2$, and the assumption $V_k Z_k \geq \log^3 n$ to guarantee that $O \left( n^{-8} \right)$ is $O \left( Z_k \right)$ in the end.
	
	Finally, let us prove~\eqref{eqn:variance_X}, for which we need to be more accurate. Roughly speaking, for $Z_k$ and $\frac{X_k}{V_k}$ small, the dominant case is again $D_k=2$ and $D_k^{(1)}=D_k^{(2)}=1$, in which case $\Delta X_k=0$. We will see that the second most likely cases are the case where $D_k=2$ and the second neighbour of $y$ has degree $3$, and the case where $D_k=3$ and the second and third neighbours of $y$ have degree $2$ each. In these two cases, we have respectively $\Delta X_k=-1$ and $\Delta X_k=1$.
	
	More precisely, let us write $\var \left( \Delta X_k | \Ff_k \right) = \E \left[ (\Delta X_k)^2 | \Ff_k \right] - \E \left[ \Delta X_k | \Ff_k \right]^2$. We first handle the second term using Proposition~\ref{prop:drift_estimates_AFP}. Using $f(z)=\frac{z^2}{2}+O(z^3)$ and $\frac{v}{x+2v+s}=\frac{1}{2}+O \left( \frac{x}{v} \right)+O \left( \frac{s}{v}\right)$, we have
	\[ \Phi_A (x,v,s)=O(z(v,s)) + O \left( \frac{x}{v} \right) + O \left( \frac{s}{v} \right) = O \left( \frac{x}{v} \right) + O \left( z(v,s) \right)\]
	by Lemma~\ref{lem:z_smooth}.
	Therefore, by Proposition~\ref{prop:drift_estimates_AFP}, we have
	\[ \E \left[ \Delta X_k | \Ff_k \right] = O \left( \frac{\log^2 n}{V_k} \right) + O \left( \frac{X_k}{V_k} \right) + O \left( Z_k \right), \]
	so
	\[ \E \left[ \Delta X_k | \Ff_k \right]^2 = O \left( \frac{\log^4 n}{V_k^2} \right) + O \left( \frac{X_k^2}{V_k^2} \right) + O \left( Z_k^2 \right).\]
	On the other hand, we denote by $A^0_k$ the event where $y$ has degree $2$ and its second neighbour as well (which implies $\Delta X_k=0$). We also denote by $A^1_k$ the event where $y$ has degree $3$ and its neighbours which are not $1$ are distinct and both have degree $2$, which implies $\Delta X_k=1$. Finally, let $A^{-1}_k$ be the event where $y$ has degree $2$ and its second neighbour has degree $3$, which implies $\Delta X_k=-1$. We have $\Delta X_k=O(D_k)$, so we can write
	\begin{equation}\label{eqn:decomposition_variance}
		\E \left[ (\Delta X_k)^2 |\Ff_k \right] = \P \left( A^1_k | \Ff_k \right) + \P \left( A^{-1}_k | \Ff_k \right) + O \left( \E \left[ D_k^2 \mathbbm{1}_{(A^0_k \cup A^1_k \cup A^{-1}_k)^c} | \Ff_k \right] \right).
	\end{equation}
	As before, we can decompose according to the values of $y$ and of its neighbours $w_1, w_2$ other than $1$ to get
	\begin{multline*}
		\P \left( A^1_k | \Ff_k \right) = \sum_{y,w_1,w_2=X_k+1}^{X_k+V_k} \P \left( \deg(y)=3, \deg(w_1)=2, \deg(w_2)=2 | \Ff_k \right) \\ \times \frac{3}{H_k-1} \frac{2}{H_k-3} \frac{2}{H_k-5},
	\end{multline*}
	where the sum is over pairwise distinct $y, w_1, w_2$. Using Lemma~\ref{lem:degrees_are_Poisson} (more precisely~\eqref{eqn:degree_distribution_three_pt_function} for $i_1=3$ and $i_2=i_3=2$), this becomes
	\begin{align*}
		\P \left( A^1_k | \Ff_k \right) &= \frac{12V_k(V_k-1)(V_k-2)}{(H_k-1)(H_k-3)(H_k-5)} \frac{Z_k^7}{24 f(Z_k)^3} \left( 1+O \left( \frac{\log^2 n}{V_k Z_k} \right)\right)\\
		&= \frac{1}{2} Z_k \left( 1+O \left( \frac{1}{V_k} \right)+ O \left( Z_k \right) + O \left( \frac{X_k}{H_k} \right) + O \left( \frac{\log^2 n}{V_k Z_k} \right) \right),
	\end{align*}
	where the second equality uses $\frac{V_k}{H_k}=\frac{1}{2}+O \left( \frac{X_k}{H_k} \right) +O(Z_k)$ and $f(Z_k)=\frac{1}{2}Z_k^2+O(Z_k^3)$.
	
	Similarly, using Lemma~\eqref{lem:degrees_are_Poisson} for $i_1=2$ and $i_2=3$, we obtain
	\begin{align*}
		\P \left( A_k^{-1} | \Ff_k \right) &= \sum_{\substack{y,w=X_k+1 \\ y \ne w}}^{X_k+V_k} \P \left( \deg(y)=2, \deg(w)=3 | \Ff_k \right) \times \frac{2}{H_k-1} \frac{3}{H_k-3}\\
		&= \frac{6V_k(V_k-1)}{(H_k-1)(H_k-3)} \frac{Z_k^5}{12f(Z_k)^2}  \left( 1+O \left( \frac{\log^2 n}{V_k Z_k} \right)\right)\\
		&= \frac{1}{2} Z_k \left( 1+O \left( \frac{1}{V_k} \right)+ O \left( Z_k \right) + O \left( \frac{X_k}{V_k} \right) + O \left( \frac{\log^2 n}{V_k Z_k} \right) \right).
	\end{align*}
	
	Finally, it remains to show that the last term of~\eqref{eqn:decomposition_variance} is small. We decompose this term according to the value of $D_k$:
	\begin{itemize}
		\item the case $D_k=1$ has probability $\frac{X_k-1}{H_k-1}$ and implies $\Delta X_k=-2$, so its contribution is $O \left( \frac{X_k}{V_k} \right)$.
		\item If $D_k=2$, then the second neighbour of $y$ has degree either $1$, or at least $4$ (if not, either $A^0_k$ or $A^{-1}_k$ occurs). The first case has probability $O \left( \frac{X_k}{H_k} \right)$ and implies $\Delta X_k=-2$, so its contribution is $O \left( \frac{X_k}{V_k} \right)$. In the second case we have $\Delta X_k=-1$, so its contribution is the probability that the second neighbour $w$ of $y$ has degree at least $4$. By summing over the possible values of $w$ and its degree, this is
		\begin{align*}
			\sum_{d \geq 4} \sum_{y,w=X_k+1}^{X_k+V_k} \P \left( \deg(y)=2 \mbox{ and } \deg(w) = d | \Ff_k \right) \times \frac{2}{H_k-1} \frac{d}{H_k-3} &= O \left( n^{-8} \right) + O \left( \sum_{d=4}^{\log n} \frac{d Z_k^d}{d! f(Z_k)}\right)\\
			&= O \left( n^{-8} \right) + O \left( Z_k^2 \right)
		\end{align*}
		by bounding $Z_k^d$ by $Z_k^4 C^{d-4}$.
		\item By a similar computation as in the previous cases, the contribution of the case $D_k \geq 4$ is bounded by
		\begin{align*}
		\E \left[ D_k^2 \mathbbm{1}_{D_k \geq 4} | \Ff_k \right] &= O \left( n^{-8} \right) + \sum_{d=4}^{\log n} \sum_{y=X_k+1}^{X_k+V_k} d^2 \P \left( \deg(y)=d | \Ff_k \right) \times \frac{d}{H_k-1} \\
		&= O \left( n^{-8} \right) + O \left( \frac{V_k}{H_k} \sum_{d=4}^{\log n} d^3 \frac{Z_k^{d}}{d! f(Z_k)} \right) \\
		&= O \left( n^{-8} \right) + O \left( \sum_{d=4}^{\log n} \frac{d^3}{d!} C^{d-4} Z_k^2\right)\\
		&= O \left( n^{-8} \right)+O(Z_k^2).
		\end{align*}
		\item Finally, if $D_k=3$, then either $y$ has a neighbour other than $1$ and $y$ with degree $\ne 2$, or it is linked twice to the same vertex of degree $2$, or it is linked to itself.
			\begin{itemize}
				\item The contribution of the case where $y$ has a neighbour $w $ other than $1$ and $y$ with degree $1$ is bounded (up to a multiplicative constant) by
				\begin{align*}
				& \P \left( D_k=3 \mbox{ and } D_k^{(1)} \geq 2 \big| \Ff_k \right) 
				\leq & \sum_{y=X_k+1}^{X_k+V_k} \sum_{w=2}^{X_k} \P \left( \deg(y)=3 |  \Ff_k \right) \times \frac{3}{H_k-1} \times \frac{2}{H_k-3},
				\end{align*}
	where the factor $2$ in the numerator accounts for the $2$ half-edges going out from $y$ that are not directed toward $1$.	Using Lemma~\ref{lem:degrees_are_Poisson}, this is $O \left( \frac{X_k V_k}{H_k^2} \frac{Z_k^3}{f(Z_k)} \right)=O \left( \frac{X_k}{V_k} \right)$.
				\item Similarly, the contribution of the case where $y$ has a neighbour $w$ other than $1$ and $y$ with degree $d \geq 3$ is bounded (up to a multiplicative constant) by
				\begin{align*}
					& \E \left[ D_k^2 \mathbbm{1}_{D_k=3 \mbox{ and } D_k^{(3)} \geq 1} \big| \Ff_k \right] \\
					\leq & \sum_{d \geq 3} \sum_{\substack{y,w=X_k+1\\ y \ne w}}^{X_k+V_k} \P \left( \deg(y)=3 \mbox{ and } \deg(w)=d | \Ff_k \right) \times \frac{3}{H_k-1} \times \frac{2d}{H_k-3},
				\end{align*}
				where the factor $2$ in the numerator again accounts for the $2$ half-edges going out from $y$ that are not directed to $1$. Using Lemma~\ref{lem:degrees_are_Poisson} in the same way as before, we find that this contribution is $O \left( Z_k^2 \right)$ (with one factor $Z_k$ coming from $y$ and another one from $w$).
				\item If $y$ is linked twice to the same vertex of degree $2$ or to itself, we have $\Delta X_k = O(1)$ so the corresponding contribution is bounded (up to a multiplicative constant) by the probability that $y$ is incident to a loop or a multiple edge. By~\eqref{eqn:very_few_loops}, this is $O \left( \frac{1}{V_k} \right)$.
			\end{itemize}
	\end{itemize}
\end{proof}

\section{Fluid limit of the Markov chain}\label{sec:fluid_limit}

\subsection{Solution to the differential equations}

Following~\cite{AFP97}, let us investigate the solution to the differential equations which appear as the limit in Proposition~\ref{prop:drift_estimates_AFP}. We recall from~\ref{eqn:defn_Phi_A},~\ref{eqn:defn_Phi_B} and~\ref{eqn:defn_Phi_C} the definition of $\Phi_A$, $\Phi_B$ and $\Phi_C$. Let $\left( \sx(t), \sv(t), \ss(t) \right)$ be the solution of the system
\[ \left\{ \begin{array}{rl}
\sx' &= \Phi_A \left( \sx, \sv, \ss \right), \\
\sv' &= \Phi_B \left( \sx, \sv, \ss \right),\\
\ss' &= \Phi_C \left( \sx, \sv, \ss \right),
\end{array}\right.\]
with the initial conditions (coming from Lemma~\ref{lem:convergence_initial_conditions})
\begin{equation}\label{eqn:initial_conditions_fluid_limit}
	\left\{ \begin{array}{rl}
		\sx(0) &= \mathrm{e}^{1- \mathrm{e}},\\
		\sv(0) &= 1-\mathrm{e}^{-\mathrm{e}}-\mathrm{e}^{1- \mathrm{e}}, \\
		\ss(0) &= \mathrm{e}+\mathrm{e}^{1- \mathrm{e}}+2\mathrm{e}^{- \mathrm{e}}-2.
	\end{array} \right.
\end{equation}
The solution is given by~\cite[Lemma 8]{AFP97} with $c= \mathrm{e}$, where $v_1$, $v$ and $2m$ in~\cite{AFP97} stand respectively for $\sx$, $\sv$ and $\sx+2\sv+\ss$. We write $\sz(t)=z \left( \sv(t), \ss(t) \right)$, where $z(v,s)$ is given by~\eqref{eqn:definition_z}, and we recall that $f(z)=\mathrm{e}^z-z-1$.
For $z \geq 0$, let $\beta(z)$ be the unique solution in $[\mathrm{e}^{-1}, +\infty)$ to the equation \begin{equation}\label{eqn:definition_beta}
\beta(z)\mathrm{e}^{ \mathrm{e}\beta(z)}=\mathrm{e}^z,
\end{equation}
and note that $\beta(0)=\mathrm{e}^{-1}$. Parametrized by $\sz$, the solutions can be written
\begin{align}
	\sv &= \beta(\sz) \mathrm{e}^{-\sz} f(\sz), \label{eqn:exact_formula_v}\\
	\sx &= \mathrm{e}^{-1} \sz^2- \sz\beta(\sz) (1-\mathrm{e}^{-\sz}), \label{eqn:exact_formula_x}\\
	\ss &= \beta(\sz) \left( \sz+\sz \mathrm{e}^{-\sz} +2 \mathrm{e}^{-\sz} -2 \right), \label{eqn:exact_formula_s}\\
	t &= 1-\beta(\sz)-\frac{1}{2\mathrm{e}} \log^2 \beta(\sz). \label{eqn:exact_formula_t}
\end{align}

In~\cite{AFP97}, it is proved in the subcritical and supercritical cases that the Markov chain $(X_k,V_k,S_k)$, once properly rescaled, converges to the process $(\sx, \sv, \ss)$. In the critical case, such a result will follow from the results of Section~\ref{sec:good_region}.

\subsection{The fluid limit near the extinction time}

As in~\cite{BCC22}, we will now focus on carefully studying the solutions $\sx, \sv, \ss$ near the extinction time $t^*=\inf \{ t \geq 0 | \sx(t)=0 \}$. We first notice that such a time exists, since plugging $\sz=0$ in~\eqref{eqn:exact_formula_t} and~\eqref{eqn:exact_formula_x} gives $t=1-\frac{3}{2\mathrm{e}}>0$ and $\sx=0$.

We now argue that we indeed have $t^*=1-\frac{3}{2\mathrm{e}}$. For this, let $z^*=\sz(t^*)$. We will show that $z^*=0$. Indeed, since $\sx(t^*)=0$, the formula~\eqref{eqn:exact_formula_x} gives
\[ \mathrm{e}^{-1} (z^*)^2-z^* \beta^* (1-\mathrm{e}^{-z^*})=0, \]
where $\beta^*=\beta(z^*)$. Hence, we have either $z^*=0$, or $\mathrm{e}^{-1} z^*- \beta^* (1-\mathrm{e}^{-z^*})=0$. In the second case, we replace $z^*$ by $ \mathrm{e} \beta^*+\log \beta^*$ using~\eqref{eqn:definition_beta}. The condition becomes
\[ \frac{1}{ \mathrm{e}} \log \beta^* +\mathrm{e}^{- \mathrm{e}\beta^*}=0.\]
We note that $\beta^*=\mathrm{e}^{-1}$ (which implies $z^*=0$) is a solution, and claim that the left-hand side is increasing in $\beta$, so the solution is unique. Indeed, differentiating the left-hand side with respect to $\beta$, we obtain $\frac{1}{ \mathrm{e}\beta}-\mathrm{e}^{1- \mathrm{e}\beta}$. We will prove that this is positive for all $\beta \ne \mathrm{e}^{-1}$, which is equivalent (via the change of variable $y=\e \beta$) to showing $y\mathrm{e}^{1-y} < 1$ for all $y \ne 1$. This last assertion is immediate by differentiating again with respect to $y$, which proves $\beta^*=\mathrm{e}^{-1}$ and $z^*=0$. Using~\eqref{eqn:exact_formula_t}, we get $t^*=1-\frac{3}{2 \mathrm{e}}$.

Moreover, by differentiating~\eqref{eqn:exact_formula_t} twice, we find that $t$ is a smooth function of $\beta$ with
\begin{equation}\label{eqn:derivatives_t_beta}
	\left. \frac{\mathrm{d}t}{\mathrm{d}\beta} \right|_{\beta=\beta^*}=0, \quad \left. \frac{\mathrm{d}^2t}{\mathrm{d}\beta^2} \right|_{\beta=\beta^*}=-2 \mathrm{e} \quad \mbox{and} \quad \left. \frac{\mathrm{d}^3t}{\mathrm{d}\beta^3} \right|_{\beta=\beta^*}=5\mathrm{e}^2,
\end{equation}
so $t^*-t = \mathrm{e} \left( \beta-\beta^* \right)^2 -\frac{5}{6} \mathrm{e}^2 \left( \beta-\beta^* \right)^3+ O \left( (\beta-\beta^*)^4 \right)$ as $t \to t^*$ (or equivalently $\beta \to \beta^*$). That is, if we write $t=t^*-\eps$, we have 
	\begin{equation}\label{eqn:expansion_beta}
		\beta=\mathrm{e}^{-1}+\mathrm{e}^{-1/2} \eps^{1/2} + \frac{5}{12} \eps +O(\eps^{3/2}).
	\end{equation}
By definition of $\beta$, we have $\left. \frac{\mathrm{d}z}{\mathrm{d}\beta} \right|_{\beta=\beta^*}=2\mathrm{e}$, so
\begin{equation}\label{eqn:asymptotic_end_z}
	\sz(t^*-\eps) \underset{\eps \to 0}{\sim} 2 \mathrm{e} \left( \beta \left( \sz(t^*-\eps) \right)-\beta^* \right) \underset{\eps \to 0}{\sim} 2\mathrm{e}^{1/2} \eps^{1/2}.
\end{equation}
We will need to estimate precisely the processes $\sx, \sv, \sz$ and their derivatives near time $t^*$. For this, we note that by replacing $z$ by $\mathrm{e}\beta+\log \beta$ in~\eqref{eqn:exact_formula_v},~\eqref{eqn:exact_formula_x},~\eqref{eqn:exact_formula_s}, these processes are explicit functions of $\beta$. In particular $\sx$ is a smooth function of $\beta$ and we can compute
\begin{equation}
\left. \frac{ \mathrm{d} \sx}{\mathrm{d} \beta} \right|_{\beta=\beta^*} = \left. \frac{\mathrm{d}^2 \sx}{\mathrm{d} \beta^2} \right|_{\beta=\beta^*} = \left. \frac{\mathrm{d}^3 \sx}{\mathrm{d} \beta^3} \right|_{\beta=\beta^*} = 0 \mbox{ and } \left. \frac{\mathrm{d}^4 \sx}{\mathrm{d} \beta^4} \right|_{\beta=\beta^*}=8\mathrm{e}^3,\label{eqn:derivatives_x_beta}
\end{equation}
so $\sx(t^*-\eps) \sim \frac{\mathrm{e}^3}{3} \left( \beta-\beta^* \right)^4$, i.e.
\begin{equation}\label{eqn:asymptotic_end_x}
	\sx(t^*-\eps) \underset{\eps \to 0}{\sim} \frac{ \mathrm{e}}{3} \eps^2.
\end{equation}
Similarly $\sv$ is a smooth function of $\beta$ with
\begin{equation}
	\left. \frac{\mathrm{d}\sv}{\mathrm{d}\beta} \right|_{\beta=\beta^*} = 0 \mbox{ and } \left. \frac{\mathrm{d}^2\sv}{\mathrm{d}\beta^2} \right|_{\beta=\beta^*} = 4 \mathrm{e} \mbox{ and } \left. \frac{\mathrm{d}^3\sv}{\mathrm{d}\beta^3} \right|_{\beta=\beta^*} = -10\mathrm{e}^2, \label{eqn:derivatives_v_beta}
\end{equation}
so, using also~\eqref{eqn:expansion_beta}:
\begin{equation}\label{eqn:asymptotic_fluid_vertices}
	\sv(t^*-\eps) = 2 \mathrm{e} \left( \beta-\beta^* \right)^2 - \frac{5}{3}\mathrm{e}^2 (\beta-\beta^*)^3 + O \left( (\beta-\beta^*)^4 \right) = 2 \eps + O \left( \eps^2 \right).
\end{equation}
We finally do the same computation for $\ss$: this is also a smooth function of $\beta$ with
\begin{equation}
	\left. \frac{ \mathrm{d} \ss}{\mathrm{d}\beta} \right|_{\beta=\beta^*} = \left. \frac{\mathrm{d}^2\ss}{\mathrm{d}\beta^2} \right|_{\beta=\beta^*} = 0 \mbox{ and } \left. \frac{\mathrm{d}^3\ss}{\mathrm{d}\beta^3} \right|_{\beta=\beta^*} = 8\mathrm{e}^2, \label{eqn:derivatives_s_beta}
\end{equation}
so
\begin{equation}\label{eqn:asymptotic_fluid_excess}
	\ss(t^*-\eps) = \frac{4}{3}\mathrm{e}^2 \left( \beta-\beta^* \right)^3 + O \left( (\beta-\beta^*)^4 \right) = \frac{4\mathrm{e}^{1/2}}{3} \eps^{3/2}+O(\eps^2)
\end{equation}
by using~\eqref{eqn:expansion_beta}.

\begin{remark}
	As mentioned before, the exponents governing the number of vertices of degrees $1$, $2$ and $3$ are the same as in the configuration model of~\cite{BCC22}. Moreover, the multiplicative constant $2$ appearing in~\eqref{eqn:asymptotic_fluid_vertices} is the same as in~\cite{BCC22} and means that in the end, most steps simply consist of removing two vertices of degree $2$. However, we note that the multiplicative constants for $\sx$ and $\ss$ are significantly different from their analogs in~\cite{BCC22}. This is due to the impact of vertices of degree $4$. For example, a step where the removed leaf has a neighbour of degree $4$ attached to three other vertices of degree $2$ will create three new leaves. By~\eqref{eqn:asymptotic_end_x}, the (normalized) number of leaves at time $t^*-\eps$ behaves like $\eps^2$. On the other hand, by combining the Poisson approximation of Lemma~\ref{lem:degrees_are_Poisson} and~\eqref{eqn:asymptotic_end_z}, the number of vertices of degree $4$ behaves like $\sv \times \sz^2 \approx \eps^2$. Hence, the number of vertices of degree $4$ is comparable to the number of leaves in the end of the process, which is why their contribution cannot be neglected.
\end{remark}

When we will estimate the drift of the Markov chain $\left( X_k, V_k, S_k \right)$, it will be important to understand the functions $\Phi_A, \Phi_B, \Phi_C$ of~\eqref{eqn:defn_Phi_A},~\eqref{eqn:defn_Phi_B},~\eqref{eqn:defn_Phi_C} in the neighbourhood of the fluid limit $\left( \sx, \sv, \ss \right)$. For this, we will rely on the following estimates on the partial derivatives of $\Phi_A, \Phi_B, \Phi_C$. For the sake of brevity, we will write $\left( \sx, \sv, \ss (t) \right)$ for $\left( \sx(t), \sv(t), \ss(t) \right)$. 

\begin{lemma}\label{lem:asymptotics_gradient_Phi}
	We have the following estimates as $\eps \to 0$:
	\begin{align*}
		\frac{\partial \Phi_A}{\partial x} \left( \sx, \sv, \ss \left( t^*-\eps \right) \right) &= -\eps^{-1} + O \left( \eps^{-1/2} \right),\\
		\frac{\partial \Phi_A}{\partial v} \left( \sx, \sv, \ss \left( t^*-\eps \right) \right) &= O \left( 1 \right),\\
		\frac{\partial \Phi_A}{\partial s} \left( \sx, \sv, \ss \left( t^*-\eps \right) \right) &= O \left( \eps^{-1/2} \right),\\ 
		\frac{\partial \Phi_B}{\partial x} \left( \sx, \sv, \ss \left( t^*-\eps \right) \right) &= O \left( \eps^{-1} \right),\\ 
		\frac{\partial \Phi_B}{\partial v} \left( \sx, \sv, \ss \left( t^*-\eps \right) \right) &= O (1),\\ 
		\frac{\partial \Phi_B}{\partial s} \left( \sx, \sv, \ss \left( t^*-\eps \right) \right) &= O \left( \eps^{-1/2} \right),\\ 
		\frac{\partial \Phi_C}{\partial x} \left( \sx, \sv, \ss \left( t^*-\eps \right) \right) &= O \left( \eps^{-1/2} \right),\\
		\frac{\partial \Phi_C}{\partial v} \left( \sx, \sv, \ss \left( t^*-\eps \right) \right) &=   \mathrm{e}^{1/2} \eps^{-1/2} + O \left( 1 \right),\\
		\frac{\partial \Phi_C}{\partial s} \left( \sx, \sv, \ss \left( t^*-\eps \right) \right) &= -\frac{3}{2} \eps^{-1} + O \left( \eps^{-1/2} \right).
	\end{align*}
\end{lemma}

In particular, the fact that $\frac{\partial \Phi_A}{\partial x}$ is negative means that $(X_k)$ enjoys a strong ``restoring force" towards its fluid limit, which will play an important role later in the paper.

\begin{proof}
	The proof is a calculation that we perform using Mathematica (see the attached computation sheet). More precisely, as an example, let us sketch the proof of the estimate on $\frac{\partial \Phi_A}{\partial s}$. We recall from~\eqref{eqn:defn_Phi_A} the formula
	\[ \Phi_A(x,v,s)=-1-\frac{x}{x+2v+s}+\frac{v^2z(v,s)^4 \mathrm{e}^{z(v,s)}}{(x+2v+s)^2 f(z(v,s))^2}-\frac{xvz(v,s)^2\mathrm{e}^{z(v,s)}}{(x+2v+s)^2f(z(v,s))},\]
	where we recall that $z(v,s)$ is given by~\eqref{eqn:definition_z}. We can rewrite $\Phi_A$ as $\Phi_A=\Psi_A \circ \varphi$, where $\varphi(x,v,s)=\left( x,v,s,z(v,s) \right)$ and
	\[ \Psi_A(x,v,s,z) = -1-\frac{x}{x+2v+s}+\frac{v^2 z^4 \mathrm{e}^z}{\left( x+2v+s \right)^2 f(z)^2} - \frac{xvz^2\mathrm{e}^z}{\left( x+2v+s \right)^2 f(z)}.  \]
	We then have
	\begin{equation}\label{eqn:decomposition_gradient_PhiA}
		\frac{\partial \Phi_A}{\partial s} = \frac{\partial \Psi_A}{\partial s} + \frac{\partial \Psi_A}{\partial z} \times \frac{\partial z}{\partial s}.
	\end{equation}
	Moreover, by~\eqref{eqn:definition_z}, we can write $s=v \left( \frac{z(\mathrm{e}^z-1)}{\mathrm{e}^z-z-1}-2 \right)$, so
	\[ \frac{\partial z}{\partial s} = \left( \frac{\partial}{\partial z} \left( v \left( \frac{z(\mathrm{e}^z-1)}{\mathrm{e}^z-z-1}-2 \right) \right) \right)^{-1}.\]
	Therefore~\eqref{eqn:decomposition_gradient_PhiA} can be expressed as a completely explicit function of $x,v,s$ and $z(v,s)$. We can finally replace $\left( x,v,s, z(v,s)\right)$ in this expression by $\left( \sx (t^*-\eps), \dots, \sz(t^*-\eps) \right)$. By the explicit expressions~\eqref{eqn:exact_formula_v} to~\eqref{eqn:exact_formula_s}, this is an explicit function of $\beta(\sz)$. We can then expand this as a power series in $\beta-\beta^*$, and finally translate the result to an expansion in $\eps$ using~\eqref{eqn:expansion_beta}, which gives the third equation of Lemma~\ref{lem:asymptotics_gradient_Phi}. The other eight estimates are proved in the exact same way. We have attached a Mathematica script that performs these calculations.
\end{proof}

\begin{lemma}\label{lem:bounds_second_partial_derivatives}
	We denote by $H\Phi(x,v,s)$ the Hessian matrix of a function $\Phi$ at the point $(x,v,s)$. Let $C>0$ be a constant and assume $x, s \leq C v$. Then we have
	\[ \| H \Phi_A(x,v,s) \|, \, \| H \Phi_B(x,v,s) \|, \, \| H \Phi_C(x,v,s) \| = O \left( \frac{1}{v^2} \right), \]
	where the implied constant only depends on $C$.	Moreover, we have
	\begin{align}
		\frac{\partial^2 \Phi_A}{\partial v^2} &= O \left( \frac{x}{v^3} \right) + O \left( \frac{s^2}{v^4} \right), \label{eqn:second_derivative_1} \\
		\frac{\partial^2 \Phi_A}{\partial v \partial s} &= O \left( \frac{x}{v^3} \right) + O \left( \frac{s}{v^3} \right), \label{eqn:second_derivative_2} \\
		\frac{\partial^2 \Phi_C}{\partial v^2} &= O \left( \frac{x}{v^3} \right) + O \left( \frac{s}{v^3} \right). \label{eqn:second_derivative_3}
	\end{align}
\end{lemma}

Later, as suggested by~\eqref{eqn:asymptotic_end_x},~\eqref{eqn:asymptotic_fluid_vertices} and~\eqref{eqn:asymptotic_fluid_excess}, we will typically apply these estimates in a regime where $x$, $v$ and $s$ are respectively of order $\eps^2$, $\eps$ and $\eps^{3/2}$. In this regime~\eqref{eqn:second_derivative_1} becomes $O \left( \eps^{-1} \right)$, whereas~\eqref{eqn:second_derivative_2} and~\eqref{eqn:second_derivative_3} become $O \left( \eps^{-3/2} \right)$ and all the other second order partial derivatives are $O \left( \eps^{-2} \right)$.

\begin{proof}
	Let us first prove that all the second order partial derivatives are $O \left( \frac{1}{v^2} \right)$.
	We write $\os=\frac{s}{v}$ and $\ox=\frac{x}{v}$, and we recall from Lemma~\ref{lem:z_smooth} that $z$ is a smooth function of $\os$. It follows from the definition~\eqref{eqn:defn_Phi_A} of $\Phi_A$ that we can write $\Phi_A(x,v,s)=\oPhi_A \left( \frac{x}{v}, \frac{s}{v} \right)$ where $\oPhi_A$ is smooth on $[0,+\infty)^2$, and similarly for $\Phi_B$ and $\Phi_C$.
	Therefore, we can write down
	\[ \frac{\partial^2 \Phi_A}{\partial x^2}(x,v,s) = \frac{1}{v^2} \frac{\partial^2 \oPhi_A}{\partial \ox^2} \left( \frac{x}{v}, \frac{s}{v} \right) = O \left( \frac{1}{v^2} \right)  \]
	since $\oPhi_A$ is smooth and $x,s=O(v)$, so the partial derivatives of $\oPhi_A$ are $O(1)$. The same argument applies to the partial derivatives $\frac{\partial^2 \Phi_A}{\partial x \partial s}$ and $\frac{\partial^2 \Phi_A}{\partial s^2}$. Similarly, we have
	\begin{equation}\label{eqn:second_derivative_vs_decomposition}
		\frac{\partial^2 \Phi_A}{\partial v \partial s}(x,v,s) = -\frac{1}{v^2} \frac{\partial \oPhi_A}{\partial \os} \left( \frac{x}{v}, \frac{s}{v} \right) - \frac{s}{v^3} \frac{\partial^2 \oPhi_A}{\partial \os^2} \left( \frac{x}{v}, \frac{s}{v} \right) - \frac{x}{v^3} \frac{\partial^2 \oPhi_A}{\partial \ox \partial \os} \left( \frac{x}{v}, \frac{s}{v} \right) = O \left( \frac{1}{v^2} \right)
	\end{equation}
	using the assumption $x,s=O(v)$, and the same computation applies to $\frac{\partial^2 \Phi_A}{\partial x \partial v}$. Finally, we have
	\begin{equation}\label{eqn:second_derivative_v_decomposition}
		\frac{\partial^2 \Phi_A}{\partial v^2} = 2\frac{x}{v^3} \frac{\partial \oPhi_A}{\partial \ox} + 2\frac{s}{v^3} \frac{\partial \oPhi_A}{\partial \os} + \frac{x^2}{v^4} \frac{\partial^2 \oPhi_A}{\partial \ox^2} + 2 \frac{xs}{v^4} \frac{\partial^2 \oPhi_A}{\partial \ox \partial \os} + \frac{s^2}{v^4} \frac{\partial^2 \oPhi_A}{\partial \os^2},
	\end{equation}
	where all the partial derivatives of $\oPhi_A$ are taken at $\left( \frac{x}{v}, \frac{s}{v} \right)$. Again, the partial derivatives are all $O(1)$, so this is $O \left( \frac{1}{v^2} \right)$, and the same argument applies to $\Phi_B$ and $\Phi_C$.
	
	We now prove the three estimates~\eqref{eqn:second_derivative_1},~\eqref{eqn:second_derivative_2} and~\eqref{eqn:second_derivative_3}. To show~\eqref{eqn:second_derivative_3}, we just need to write down the analog of~\eqref{eqn:second_derivative_v_decomposition} for $\Phi_C$ instead of $\Phi_A$. As before, all the partial derivatives of $\oPhi_C$ are $O(1)$, so using $x,s=O(v)$, all terms are either $O \left( \frac{x}{v^3} \right)$ or $O \left( \frac{s}{v^3} \right)$.
	
	Moreover, we notice that all the terms of~\eqref{eqn:second_derivative_vs_decomposition} are $O \left( \frac{x}{v^3} \right)$ or $O \left( \frac{s}{v^3} \right)$, except perhaps $\frac{1}{v^2} \frac{\partial \oPhi_A}{\partial \os} \left( \frac{x}{v}, \frac{s}{v} \right)$. Similarly, all the terms of~\eqref{eqn:second_derivative_v_decomposition} are $O \left( \frac{x}{v^3} \right)$ or $O \left( \frac{s^2}{v^4} \right)$, except perhaps $2 \frac{s}{v^3} \frac{\partial \oPhi_A}{\partial \os} \left( \frac{x}{v}, \frac{s}{v} \right)$. Therefore, to prove both~\eqref{eqn:second_derivative_1} and~\eqref{eqn:second_derivative_2}, it is sufficient to prove that for $\os, \ox=O(1)$, we have
	\[ \frac{\partial \oPhi_A}{\partial \os}(\ox,\os) = O \left( \os+\ox \right).\]
	Since $\oPhi_A$ is smooth, the partial derivative $\frac{\partial \oPhi_A}{\partial \os}$ is Lipschitz, so it is sufficient to show $\frac{\partial \oPhi_A}{\partial \os}(0,0)=0$, which is a direct computation that we do using Mathematica (see the second part of the attached computation sheet). 
\end{proof}

Finally, the last purely continuous estimate that we will need is on the second derivative of the fluid limit of our Markov chain.

\begin{lemma}\label{lem:second_derivatives}
	As $\eps \to 0$, we have
	\begin{align}
		\sx''(t^*-\eps) &= O(1), \label{eqn:second_derivative_x_asymptotic}\\
		\sv''(t^*-\eps) &= O \left( \eps^{-1/2} \right), \label{eqn:second_derivative_v_asymptotic} \\
		\ss''(t^*-\eps) &= O \left( \eps^{-1/2} \right). \label{eqn:second_derivative_s_asymptotic}
	\end{align}
\end{lemma}

\begin{proof}
	By~\eqref{eqn:exact_formula_x} (combined with the equation $z= \mathrm{e}\beta(z)+\log \beta(z)$) and~\eqref{eqn:exact_formula_t}, both $\sx (t)$ and $t$ can be written as explicit functions of $\beta \left( \sz(t) \right)$. Therefore, we can write
	\[ \sx'(t) = \left( \frac{\mathrm{d} t}{\mathrm{d} \beta} \right)^{-1} \times \frac{\mathrm{d} \sx}{\mathrm{d} \beta} \]
	and
	\begin{equation}\label{eqn:second_derivative_x_expanded}
		\sx''(t) = \left( \frac{\mathrm{d} t}{\mathrm{d} \beta} \right)^{-2} \times \frac{\mathrm{d}^2 \sx}{\mathrm{d} \beta^2} - \left( \frac{\mathrm{d} t}{\mathrm{d} \beta} \right)^{-3} \times \frac{\mathrm{d}^2 t}{\mathrm{d} \beta^2} \times \frac{\mathrm{d} \sx}{\mathrm{d} \beta}.
	\end{equation}
	By~\eqref{eqn:derivatives_x_beta}, we have $\frac{\mathrm{d} \sx}{\mathrm{d} \beta}=O \left( (\beta-\beta^*)^3 \right)$ and $\frac{\mathrm{d}^2 \sx}{\mathrm{d} \beta^2}=O \left( (\beta-\beta^*)^2 \right)$. On the other hand, by~\eqref{eqn:derivatives_t_beta}, we have $\frac{\mathrm{d} t}{\mathrm{d} \beta} \sim 2 \mathrm{e}(\beta-\beta^*)$ as $\beta \to \beta^*$ and $\frac{\mathrm{d}^2 t}{\mathrm{d} \beta^2}=O(1)$. Combining all these estimates, we find that~\eqref{eqn:second_derivative_x_expanded} is $O(1)$.
	
	The proof of~\eqref{eqn:second_derivative_v_asymptotic} and~\eqref{eqn:second_derivative_s_asymptotic} is similar: we can write the same equation as~\eqref{eqn:second_derivative_x_expanded} for $\sv$ and $\ss$. Using~\eqref{eqn:derivatives_t_beta} and~\eqref{eqn:derivatives_v_beta}, we obtain
	\begin{align*}
		\left( \frac{\mathrm{d} t}{\mathrm{d} \beta} \right)^{-2} \times \frac{\mathrm{d}^2 \sv}{\mathrm{d} \beta^2} &= \mathrm{e}^{-1} (\beta-\beta^*)^{-2} + O \left( (\beta-\beta^*)^{-1} \right),\\
		\left( \frac{\mathrm{d} t}{\mathrm{d} \beta} \right)^{-3} \times \frac{\mathrm{d}^2 t}{\mathrm{d} \beta^2} \times \frac{\mathrm{d} \sv}{\mathrm{d} \beta} &= \mathrm{e}^{-1} (\beta-\beta^*)^{-2} + O \left( (\beta-\beta^*)^{-1} \right),
	\end{align*}
	so $\sv''(t^*-\eps)=O \left( (\beta-\beta^*)^{-1} \right)=O \left( \eps^{-1/2} \right)$ by~\eqref{eqn:expansion_beta}. For the estimate on $\ss$, using~\eqref{eqn:derivatives_s_beta}, we find that both terms are $O \left( (\beta-\beta^*)^{-1} \right)=O \left( \eps^{-1/2} \right)$, which proves~\eqref{eqn:second_derivative_s_asymptotic}.
\end{proof}

\section{Drift and variance estimates in the neighbourhood of the fluid limit}\label{sec:drift_variance}

Our goal is now to combine the drift and variance estimates of Section~\ref{sec:Markov_chain} with the fluid limit computations of Section~\ref{sec:fluid_limit} to estimate the drift and variance of the Markov chain $\left( X, V, S \right)$ in the neighbourhood of the fluid limit. While Propositions~\ref{prop:drift_estimates_AFP} and~\ref{prop_variance_estimates_v1} consisted mostly of refining result from~\cite{AFP97}, the proofs here will be very similar to those of~\cite[Section 4.1]{BCC22}.

Just like in~\cite{BCC22}, we will spend most of the rest of the paper studying the fluctuations of the three processes $(X_k, V_k, S_k)$ around their fluid limit described in Section~\ref{sec:fluid_limit}. We recall that $\theta^n$ is the first time $k$ such that $X_k=0$ and that $t^*$ is the first time $t$ where $\sx(t)=0$. For $0 \leq k \leq \theta^n$, we define
$$ \left\{\begin{array}{l} A_k = X_k - n\mathscr{X} \left( \frac{k}{n}\right),  \\  B_k = V_k - n\mathscr{V} \left( \frac{k}{n}\right), \\ C_k = S_k - n \mathscr{S} \left( \frac{k}{n}\right) .\end{array}\right.$$

Just like in~\cite{BCC22}, it will be crucial for us to understand the order of magnitude of the fluctuations $A, B, C$ near the extinction time $\theta^n$. For technical reasons, we will work with the stopping time
\[ \Tt^n := \theta^n \wedge \left( t^* n-n^{3/5-1/100} \right),\]
and we will show later that $\Tt^n=\theta^n$ with high probability.
To measure how close we are to the end, for all $0 \leq k \leq \Tt^n$, we write $\eps_k=\frac{t^* n-k}{n}$ so that $k=(t^*-\eps_k)n$, and we note that $\eps_k \geq n^{-2/5-1/100}$. For $0 \leq k \leq \Tt^n$, we also define the rescaled fluctuations 
$$ \left\{\begin{array}{l} \At_k = \frac{A_k}{\eps_k^{3/4}\sqrt{n}},  \\ 
\Bt_k = \frac{B_k}{\sqrt{n}},\\
\Ct_k = \frac{C_k}{\eps_k^{1/2}\sqrt{n}}.\end{array}\right.$$
Note that the normalizations are the same as in~\cite{BCC22}, with $C$ playing the same role as the fluctuations on the number of vertices of degree $3$. One of the goals of Section~\ref{sec:good_region} will be to prove that these rescaled fluctuations stay of order at most $\log n$ all along the process. For now, our goal is to show the following drift and variance estimates, which are the natural analogs of~\cite[Propositions 4 and 5]{BCC22}.

\begin{proposition}[Drift estimates]\label{prop:drift_estimates}
	For any $0 \leq k < \Tt^n$,  
	if we have $|\widetilde{A}_k|, |\widetilde{B}_k|, |\widetilde{C}_k|<1000 \log n$, then
	\begin{align}
		\E \left[ \Delta \At_k | \Ff_k \right] &= -\frac{1}{4} \frac{1}{\eps_k n} \At_k + O \left( \frac{\eps_k^{1/2}}{\eps_k n}|\At_k| \right) + O \left( \frac{ \eps_k^{1/4}}{\eps_k n} \max \left( |\Bt_k|, |\Ct_k| \right) \right) + O  \left( \frac{n^{-1/30}}{\eps_k n} \right), \label{eqn_drift_estimate_A} \\
		\E \left[ \Delta \Bt_k | \Ff_k \right] &= O \left( \frac{\eps_k^{3/4}}{\eps_k n} \max \left( |\At_k|, |\Bt_k|, |\Ct_k| \right) \right) + O \left( \frac{n^{-1/30}}{\eps_k n} \right), \label{eqn_drift_estimate_B} \\
		\E \left[ \Delta \Ct_k | \Ff_k \right] &= \frac{1}{\eps_k n} \left(  \sqrt{ \mathrm{e}}\Bt_k - \Ct_k \right) + O \left( \frac{\eps_k^{1/2}}{\eps_k n} \max \left( |\Bt_k|, |\Ct_k| \right) \right) + O \left( \frac{\eps_k^{3/4}}{\eps_k n} |\At_k| \right) + \left( \frac{n^{-1/30}}{\eps_k n} \right), \label{eqn_drift_estimate_C}
	\end{align}
	where the constants implied by the $O$ notation are absolute.
\end{proposition}

Note that we have displayed all the results with a denominator $\eps_k n$. This will be convenient later to sum up the contributions of values of $k$ at the same ``scale", i.e. where $k$ ranges from $\left( t^*-\eps \right) n$ to $\left( t^*-\frac{\eps}{2} \right) n$ for some $\eps>0$. We also highlight that $\widetilde{A}$ exhibits a strong negative feedback, which will be very useful to guarantee that it stays small.

\begin{proposition}[Variance estimates]\label{prop:variance_estimates}
	For any $0 \leq k < \Tt^n$,  
	if we have $|\widetilde{A}_k|, |\widetilde{B}_k|, |\widetilde{C}_k|<1000 \log n$, then
	\begin{align}
	\var \left( \Delta \widetilde{A}_k | \mathcal{F}_k \right) &= \frac{2\sqrt{ \mathrm{e}}}{\eps_k n}  + O \left( \frac{n^{-1/30}}{\eps_k n} \right) + O \left( \frac{\eps_k^{1/2} n^{1/100}}{\eps_k n} \right), \label{eqn_variance_estimate_A}  \\
	\var \left( \Delta \widetilde{A}_k | \mathcal{F}_k \right) &= O \left( \frac{1}{\eps_k n} \right), \label{eqn_variance_estimate_A_bis} \\
	\var \left(\Delta \widetilde{B}_k | \mathcal{F}_k \right) &= O \left( \frac{\eps_k}{\eps_k n} \right), \label{eqn_variance_estimate_B} \\
	\var \left( \Delta \widetilde{C}_k | \mathcal{F}_k \right) &= O \left( \frac{\eps_k^{1/2}}{\eps_k n} \right), \label{eqn_variance_estimate_C}
	\end{align}
	where the constants implied by the $O$ notation are absolute.
\end{proposition}

During most of the process, we will only rely on the ``crude" estimates~\eqref{eqn_variance_estimate_A_bis},~\eqref{eqn_variance_estimate_B} and~\eqref{eqn_variance_estimate_C}. We will only need the more accurate estimate~\eqref{eqn_variance_estimate_A} in the final phase when $\eps_k$ is of order $n^{-2/5}$ and the fluctuations $A_k$ become comparable to $X_k$. In particular, in this regime, the error term $\frac{\eps_k^{1/2} n^{1/100}}{\eps_k n}$ will be much smaller than $\frac{1}{\eps_k n}$.

As in~\cite[Section 4.1]{BCC22}, we first state an easy lemma which will allow us to roughly estimate some denominators and error terms in the computations to prove Propositions~\ref{prop:drift_estimates} and~\ref{prop:variance_estimates}.

\begin{lemma}\label{lem_XVH_not_too_small}
	There are absolute constants $C,c>0$ such that for $n$ large enough, if $0 \leq k < \Tt^n$ satisfies $|\widetilde{A}_k|, |\widetilde{B}_k|, |\widetilde{C}_k|<1000 \log n$, then
	\begin{equation}\label{eqn:crude_bounds_XVS}
		X_k \leq C \min \left( \eps_k^2 n \times n^{1/100}, \eps_k^{3/2} n \right), \quad c \eps_k n \leq V_k \leq C \eps_k n, \quad S_k \leq C \eps_k^{3/2} n
	\end{equation}
	and
	\begin{equation}\label{eqn:general_bound_Z}
		c \eps_k^{1/2} \leq Z_k \leq C \eps_k^{1/2}.
	\end{equation}
\end{lemma}

	The reason why we needed to write down two different upper bounds on $X_k$ is that the first one is too crude for our purpose when $\eps_k$ is of order $1$, whereas the second one is too crude when $\eps_k$ becomes very small. We note that, although the fluid limit approximation gives $X_k \approx \eps_k^2 n$, it is not possible to write down $X_k=O \left( \eps_k^2 n \right)$ in general. Indeed, the final regime that we will be interested in is precisely the regime where the fluctuations of $X_k$ start to overwhelm the fluid limit.

\begin{proof}
	Note that $k <\Tt^n$ guarantees $\eps_k \geq n^{-2/5-1/100}$. The proofs of the estimates on $V_k, S_k$  and of $X_k \leq C \eps_k^2n \times n^{1/100}$ are exactly the same as for~\cite[Lemma 1]{BCC22}, with $V_k$ (resp. $S_k$) playing the same role as the $Y_k$ (resp. $Z_k$) of~\cite{BCC22}. Moreover, we can write
	\[ X_k = n \sx \left( \frac{k}{n} \right) + \eps_k^{3/4} \sqrt{n} \At_k = O \left( \eps_k^2 n \right) + O \left( \eps_k^{3/4} \sqrt{n} \log n \right), \]
	and the second term is $O \left( \eps_k^{3/2} n \right)$ because $\eps_k \geq n^{-2/5-1/100}$.
	
	The only bounds left to prove are those on $Z_k$. We first recall from~\eqref{eqn:definition_z} the definition of $Z_k$:
	\[\frac{Z_k(\mathrm{e}^{Z_k}-1)}{\mathrm{e}^{Z_k}-Z_k-1}=2+\frac{S_k}{V_k}.\]
	By~\eqref{eqn:crude_bounds_XVS}, the right-hand side is bounded by an absolute constant, and so is $Z_k$. Moreover, the left-hand side is $2+\frac{1}{3}Z_k+O \left( Z_k^2 \right)$, so there are two absolute constants $c,C$ such that
	\begin{equation}\label{eqn:crude_bound_Z}
		c \frac{S_k}{V_k} \leq Z_k \leq C \frac{S_k}{V_k}.
	\end{equation}
	By definition of $C_k$, we have $S_k=C_k+n \mathscr{S} \left( \frac{k}{n} \right)$, which is of order $\eps_k^{3/2} n$ by~\eqref{eqn:asymptotic_fluid_excess} and the assumption on $|\Ct_k|$. On the other hand, by~\eqref{eqn:crude_bounds_XVS} the denominator of~\eqref{eqn:crude_bound_Z} is of order $\eps_k n$ and the bounds on $Z_k$ follow.
\end{proof}

\begin{proof}[Proof of Proposition~\ref{prop:drift_estimates}]
	Before starting the proof, we notice that the functions $\Phi_A, \Phi_B, \Phi_C$ are homogeneous, so we can rewrite Proposition~\ref{prop:drift_estimates_AFP} as
	\begin{align}
	\E \left[ \Delta X_k | \Ff_k \right] &= \Phi_A \left( \frac{X_k}{n}, \frac{V_k}{n}, \frac{S_k}{n} \right) + O \left( \frac{\log^2 n}{V_k} \right), \label{eqn:drift_X_refined}\\
	\E \left[ \Delta V_k | \Ff_k \right] &= \Phi_B \left( \frac{X_k}{n}, \frac{V_k}{n}, \frac{S_k}{n} \right) + O \left( \frac{\log^2 n}{V_k Z_k} \right), \nonumber \\
	\E \left[ \Delta S_k | \Ff_k \right] &= \Phi_C \left( \frac{X_k}{n}, \frac{V_k}{n}, \frac{S_k}{n} \right) + O \left( \frac{\log^2 n}{V_k Z_k} \right). \nonumber
	\end{align}
	We also notice that under the assumptions of Proposition~\ref{prop:drift_estimates}, the assumptions of Proposition~\ref{prop:drift_estimates_AFP} are satisfied for some absolute constant $C>0$. Indeed, we have $X_k>0$ beacuse $k<\Tt^n$ and Lemma~\ref{lem_XVH_not_too_small} gives $Z_k=O(1)$ and $V_k Z_k \geq c \eps_k^{3/2} n \geq \log^3 n$ if $n$ is large enough because $\eps_k \geq n^{-2/5-1/100}$.
		
	Let us now start with the drift estimate on $\At$, which is the most complex one. As in~\cite{BCC22}, the left-hand side of~\eqref{eqn_drift_estimate_A} can be decomposed as
	\begin{align}
	&\nonumber \left| \E \left[ \Delta \widetilde{A}_k | \mathcal{F}_k \right] +\frac{1}{\eps_k n} \frac{1}{4} \widetilde{A}_k \right|\\
	& \leq \left| \left( \frac{\eps_k^{3/4}}{\eps_{k+1}^{3/4}} - 1 \right) \widetilde{A}_k - \frac{3}{4} \frac{1}{\eps_k n} \widetilde{A}_k  \right| \label{eqn_drift_estimate_A_term_1} \\
	&+ \frac{1}{\eps_{k+1}^{3/4} \sqrt{n}} \left| \E \left[ \Delta X_k | \mathcal{F}_k \right] - \Phi_A \left( \frac{X_k}{n}, \frac{V_k}{n}, \frac{S_k}{n} \right) \right| \label{eqn_drift_estimate_A_term_2}\\
	&+ \frac{1}{\eps_{k+1}^{3/4} \sqrt{n}} \left| \Phi_A \left( \frac{X_k}{n}, \frac{V_k}{n}, \frac{S_k}{n} \right) - \Phi_A \left( \mathscr{X}, \mathscr{V}, \mathscr{S} \left( \frac{k}{n} \right) \right) - \left( \frac{A_k}{n}, \frac{B_k}{n}, \frac{C_k}{n} \right) \cdot \nabla \Phi_A \left( \mathscr{X}, \mathscr{V}, \mathscr{S} \left( \frac{k}{n} \right) \right) \right| \label{eqn_drift_estimate_A_term_3}\\
	&+ \frac{1}{\eps_{k+1}^{3/4} \sqrt{n}} \left|  \left( \frac{A_k}{n}, \frac{B_k}{n}, \frac{C_k}{n} \right) \cdot \nabla \Phi_A \left( \mathscr{X}, \mathscr{V}, \mathscr{S} \left( \frac{k}{n} \right) \right) - \left( - \frac{A_k}{\eps_k n} \right) \right| \label{eqn_drift_estimate_A_term_4}\\
	&+ \frac{1}{\eps_{k+1}^{3/4} \sqrt{n}} \left| \Phi_A \left( \mathscr{X}, \mathscr{V}, \mathscr{S} \left( \frac{k}{n} \right) \right) - n \left( \mathscr{X} \left( \frac{k+1}{n} \right) - \mathscr{X} \left( \frac{k}{n} \right) \right) \right|. \label{eqn_drift_estimate_A_term_5}
	\end{align}
	We will handle each of those terms one by one. Just like in~\cite{BCC22}, the first term~\eqref{eqn_drift_estimate_A_term_1} is $O \left( \frac{\At_k}{(\eps_k n)^2} \right)$, so it is $O \left( \frac{n^{-1/30}}{\eps_k n} \right)$. By\footnote{Note that this is where we need the improved error term that we obtained in~\eqref{eqn:drift_v1_X}. The bound provided by~\cite{AFP97} would not be sufficient for our purpose here, but it will be enough for the analog estimates on the drifts of $\Bt$ and $\Ct$.} Proposition~\ref{prop:drift_estimates_AFP}, the second term~\eqref{eqn_drift_estimate_A_term_2} is $O \left( \frac{1}{\eps_k^{3/4}\sqrt{n}} \times \frac{\log^2 n}{V_k} \right)$. Using Lemma~\ref{lem_XVH_not_too_small} and $\eps_k \geq n^{-2/5-1/100}$, this is $O \left( \frac{n^{-1/30}}{\eps_k n} \right)$.
	
	We move on to the third term~\eqref{eqn_drift_estimate_A_term_3}: as in~\cite{BCC22}, this is a linear approximation of $\Phi_A$ near $\left( \sx, \sv, \ss \left( \frac{k}{n} \right) \right)$, so the bound relies on bounding second derivatives of $\Phi_A$, which is done in Lemma~\ref{lem:bounds_second_partial_derivatives}. More precisely~\eqref{eqn_drift_estimate_A_term_3} is bounded by
	\begin{equation}\label{eqn:decomposition_of_linear_approximation}
		\frac{1}{\eps_{k+1}^{3/4} \sqrt{n}} \sum_{1 \leq i,j \leq 3} |w_i| \times |w_j| \times \max_{\substack{|u_1-\sx(k/n)| \leq |w_1| \\ |u_2-\sv(k/n)| \leq |w_2| \\ |u_3-\ss(k/n)| \leq |w_3|}} \left| \frac{\partial^2 \Phi_A}{\partial x_i \partial x_j}(u_1,u_2,u_3)  \right|,
	\end{equation}
	where we wrote $(w_1,w_2,w_3)=\left( \frac{A_k}{n}, \frac{B_k}{n}, \frac{C_k}{n} \right)$ for the sake of brevity and $u_1, u_2, u_3$ stand respectively for the variables $x, v, s$. By the assumption $|\widetilde{A}|, |\widetilde{B}|, |\widetilde{C}|<1000 \log n$, we have the bounds
	\begin{equation}\label{eqn:bd_w1w2w3}
	|w_1| \leq 1000 \, \eps_k^{3/4} \frac{\log n}{\sqrt{n}}, \quad |w_2| \leq 1000 \frac{\log n}{\sqrt{n}}, \quad |w_3| \leq 1000 \, \eps_k^{1/2} \frac{\log n}{\sqrt{n}}.
	\end{equation}
	On the other hand, we can use Lemma~\ref{lem:bounds_second_partial_derivatives} to bound one by one the $9$ partial derivatives, and handle the estimates given by Lemma~\ref{lem:bounds_second_partial_derivatives} using Lemma~\ref{lem_XVH_not_too_small}.
	More precisely, by Lemma~\ref{lem_XVH_not_too_small}, we know that $u_1, u_2, u_3$ are respectively $O \left( \eps_k^2 \times n^{1/100} \right)$, of order $\eps_k n$ and $O \left( \eps_k^{3/2} \times n^{1/100} \right)$.
	Therefore, the terms $O \left( \frac{1}{v^2} \right)$ in the bounds of Lemma~\ref{lem:bounds_second_partial_derivatives} become $O(\eps_k^{-2} \times n^{1/50})$. Similarly, the terms $O \left( \frac{s}{v^3} \right)$, $O \left( \frac{x}{v^3} \right)$ and $O \left( \frac{s^2}{v^4} \right)$ are respectively $O(\eps_k^{-3/2}  \times n^{1/50})$, $O(\eps_k^{-1} \times n^{1/50})$ and $O(\eps_k^{-1}  \times n^{1/50})$. Combining these bounds with~\eqref{eqn:bd_w1w2w3}, we can bound one by one the terms of~\eqref{eqn:decomposition_of_linear_approximation}. For example, the term $i=j=2$ is
	\[ \frac{1}{\eps_{k+1}^{3/4}\sqrt{n}} \times O \left( \frac{\log n}{\sqrt{n}} \right) \times O \left( \frac{\log n}{\sqrt{n}} \right) \times O \left( \eps_k^{-1} \times n^{1/50} \right) = O \left( \frac{\eps_k^{-3/4} n^{-1/2} n^{1/50} \log^2 n}{\eps_k n} \right).  \]
	Using the assumption $\eps_k \geq n^{-2/5-1/100}$, we find that the numerator is $O \left( n^{-1/30} \right)$, so this term is $O \left( \frac{n^{-1/30}}{\eps_k n} \right)$. By a similar computation, we prove the same thing for all terms of~\eqref{eqn:decomposition_of_linear_approximation}.
	
	For the fourth term~\eqref{eqn_drift_estimate_A_term_4}, we first write
	\begin{equation}\label{eqn:gradient_to_partial_derivatives}
		\left( \frac{A_k}{n}, \frac{B_k}{n}, \frac{C_k}{n} \right) \cdot \nabla \Phi_A \left( \mathscr{X}, \mathscr{V}, \mathscr{S} \left( \frac{k}{n} \right) \right) = \frac{A_k}{n} \frac{\partial \Phi_A}{\partial x} + \frac{B_k}{n} \frac{\partial \Phi_A}{\partial v} + \frac{C_k}{n} \frac{\partial \Phi_A}{\partial s},
	\end{equation}
	where the partial derivatives are always considered at the point $\left( \sx, \sv, \ss \left( \frac{k}{n} \right) \right)$. We then use the first three equations of Lemma~\ref{lem:asymptotics_gradient_Phi} to estimate the partial derivatives of $\Phi_A$. In particular, by the first equation of Lemma~\ref{lem:asymptotics_gradient_Phi}, the first term is $-\frac{A_k}{\eps_k n}+O \left(\frac{\eps_k^{1/2}|A_k|}{\eps_k n} \right)$. Similarly, by the second equation of Lemma~\ref{lem:asymptotics_gradient_Phi}, the second term of~\eqref{eqn:gradient_to_partial_derivatives} is $O \left( \frac{B_k}{n} \right) = O \left( \frac{|\Bt_k|}{\sqrt{n}} \right)$.
	Finally, by the third equation of Lemma~\ref{lem:asymptotics_gradient_Phi} and the definition of $\Ct_k$, the third term of~\eqref{eqn:gradient_to_partial_derivatives} is
	\[ O \left( \frac{|C_k|}{n} \times \eps_k^{-1/2} \right) = O \left( \frac{|\Ct_k|}{\sqrt{n}} \right). \]
	Combining our estimates on the three terms of~\eqref{eqn:gradient_to_partial_derivatives}, we finally obtain that~\eqref{eqn_drift_estimate_A_term_4} is
	\begin{multline*}
	O \left( \frac{1}{\eps_k^{3/4} \sqrt{n}} \times \frac{\eps_k^{1/2}|A_k|}{\eps_k n} \right) + O \left( \frac{1}{\eps_k^{3/4} \sqrt{n}} \times \frac{\max(|\Bt_k|,|\Ct_k|)}{\sqrt{n}} \right)\\ = O \left( \frac{\eps_k^{1/2} |\At_k|}{\eps_k n} \right) + O \left( \frac{\eps_k^{1/4}}{\eps_k n} \max \left( |\Bt_k|,|\Ct_k| \right) \right).
	\end{multline*}
	Finally, since $\sx'=\Phi_A(\sx,\sv,\ss)$, the fifth term~\eqref{eqn_drift_estimate_A_term_5} is just a linear approximation of the function $\sx$, so it is bounded by
	\begin{equation}\label{eqn:bound_fifth_term_second_derivative}
		\frac{n}{\eps_{k+1}^{3/4} \sqrt{n}} \times \left( \frac{1}{n}\right)^2 \times \max_{\left[ \frac{k}{n}, \frac{k+1}{n} \right]} \left| \sx'' \right|.
	\end{equation}
	Using Lemma~\ref{lem:second_derivatives}, this is $O \left( n^{-3/2} \eps_{k+1}^{-3/4} \right)$, which is $O \left( \frac{n^{-1/30}}{\eps_k n} \right)$ because $\eps_k \geq n^{-2/5-1/100}$.
	
	Just like in~\cite{BCC22}, the proofs of the drift estimates on $\Bt_k$ and $\Ct_k$ follow the same lines. More precisely, the decomposition of the error into five terms is the same up to the following changes:
	\begin{itemize}
		\item in the first term~\eqref{eqn_drift_estimate_A_term_1}, the constant $\frac{3}{4}$ becomes $\frac{1}{2}$ for $\Ct$, and the term disappears completely for $\Bt$;
		\item in the terms~\eqref{eqn_drift_estimate_A_term_2}, \eqref{eqn_drift_estimate_A_term_3}, \eqref{eqn_drift_estimate_A_term_4} and~\eqref{eqn_drift_estimate_A_term_5}, the factors $\frac{1}{\eps_{k+1}^{3/4}\sqrt{n}}$ become $\frac{1}{\sqrt{n}}$ for $\widetilde{B}$ and $\frac{1}{\eps_{k+1}^{1/2}\sqrt{n}}$ for $\widetilde{C}$;
		\item in the fourth term~\eqref{eqn_drift_estimate_A_term_4}, the drift $-\frac{A_k}{\eps_k n}$ becomes $0$ for $\widetilde{B}$ and $  \mathrm{e}^{1/2} \frac{\eps_k^{1/2} B_k}{\eps_k n} - \frac{3}{2} \frac{C_k}{\eps_k n}$ for $\widetilde{C}$.
	\end{itemize}

	The analog of the first term~\eqref{eqn_drift_estimate_A_term_1} is then $O \left( \frac{n^{-1/30}}{\eps_k n} \right)$ for the same reason as for $\At$.
	
	For the second term~\eqref{eqn_drift_estimate_A_term_2}, we also rely on Proposition~\ref{prop:drift_estimates_AFP}. We note that the bound given by Proposition~\ref{prop:drift_estimates_AFP} is slightly weaker because of the factor $Z_k$ in the denominator, which is compensated by the substitution of the factor $\frac{1}{\eps_{k+1}^{3/4} \sqrt{n}}$ by $\frac{1}{\eps_{k+1}^{1/2}\sqrt{n}}$ or $\frac{1}{\sqrt{n}}$ in the analog of~\eqref{eqn_drift_estimate_A_term_2}. More precisely, by Proposition~\ref{prop:drift_estimates_AFP}, the analog of~\eqref{eqn_drift_estimate_A_term_2} for $\Ct$ is
	\[ O \left( \frac{1}{\eps_{k+1}^{1/2} \sqrt{n}} \times \frac{\log^2 n}{V_k Z_k} \right) = O \left( \frac{\log^2 n}{\eps_k^2 n^{3/2}} \right) = O \left( \frac{n^{-1/30}}{\eps_k n} \right) \]
	by using Lemma~\ref{lem_XVH_not_too_small} and then $\eps_k \geq n^{-2/5-1/100}$.
	
	For the third term~\eqref{eqn_drift_estimate_A_term_3}, the proof is the same as for $\At$: we decompose the error using all the second-order partial derivatives of $\Phi_B$ and $\Phi_C$ and estimate the corresponding errors one by one. We find that all error terms are $O \left( \frac{n^{-1/30}}{\eps_k n} \right)$.
	
	For the fourth term~\eqref{eqn_drift_estimate_A_term_4}, we write down the natural analog of~\eqref{eqn:gradient_to_partial_derivatives} for $\Phi_B$ and $\Phi_C$ and use Lemma~\ref{lem:asymptotics_gradient_Phi} for each of the three partial derivatives. For $\Phi_C$, we find
	\[ \frac{1}{\eps_{k+1}^{1/2}\sqrt{n}} \left( \frac{A_k}{n}, \frac{B_k}{n}, \frac{C_k}{n} \right) \cdot \nabla \Phi_C \left( \sx, \sv, \ss \left( \frac{k}{n} \right) \right) = \frac{1}{\eps_{k+1}^{1/2}\sqrt{n}} \left( \frac{A_k}{n} \frac{\partial \Phi_C}{\partial x} + \frac{B_k}{n} \frac{\partial \Phi_C}{\partial v} + \frac{C_k}{n} \frac{\partial \Phi_C}{\partial s} \right), \]
	where all the partial derivatives are taken at the point $\left( \sx, \sv, \ss \left( \frac{k}{n} \right) \right)$. Using Lemma~\ref{lem:asymptotics_gradient_Phi} and replacing $A_k$,  $B_k$ and $C_k$ by $\eps_k^{3/4} \sqrt{n} \At_k$, $\sqrt{n} \Bt_n$ and $\eps_k^{1/2} \sqrt{n} \Ct_k$, we find that the first term is $O \left( \frac{\eps_k^{3/4} |\At_k|}{\eps_k n} \right)$, the second term is $  \mathrm{e}^{1/2} \frac{\Bt_k}{\eps_k n} + O \left( \frac{\eps_k^{1/2} |\Bt_k|}{\eps_k n} \right)$ and the third term is $-\frac{3}{2} \frac{\Ct_k}{\eps_k n} + O \left( \frac{\eps_k^{1/2} |\Ct_k|}{\eps_k n} \right)$. The argument for $\Phi_B$ is similar, and we find that the three terms are respectively $O \left( \frac{\eps_k^{3/4} |\At_k|}{\eps_k n} \right)$, $O \left( \frac{\eps_k |\Bt_k|}{\eps_k n} \right)$ and $O \left( \frac{\eps_k |\Ct_k|}{\eps_k n} \right)$.
	
	Finally, for the fifth term~\eqref{eqn_drift_estimate_A_term_5}, we quantify the linear approximation using the second derivatives of $\sv$ and $\ss$, which we bound using Lemma~\ref{lem:second_derivatives}. We find again that the analog of this term is always $O \left( \frac{n^{-1/30}}{\eps_k n} \right)$.
\end{proof}

We now move on to the proof of the variance estimates of Proposition~\ref{prop:variance_estimates}, which relies mostly on Proposition~\ref{prop_variance_estimates_v1}.

\begin{proof}[Proof of Proposition~\ref{prop:variance_estimates}]
	Just like in the proof of Proposition~\ref{prop:drift_estimates}, we first notice that Lemma~\ref{lem_XVH_not_too_small} guarantees that the assumptions of Proposition~\ref{prop_variance_estimates_v1} are satisfied under the assumptions of Proposition~\ref{prop:variance_estimates}. Therefore, by Proposition~\ref{prop_variance_estimates_v1}, we can write
	\begin{align}
		\var \left( \Delta \At_k | \Ff_k \right) &= \frac{1}{\eps_k^{3/2} n} \var \left( \Delta X_k | \Ff_k \right) \nonumber \\
		&= \frac{Z_k}{\eps_k^{3/2}n} + O \left( \frac{Z_k^2}{\eps_k^{3/2}n} \right) + O \left( \frac{1}{\eps_k^{3/2}n} \frac{X_k}{V_k} \right) + O \left( \frac{1}{\eps_k^{3/2}n} \frac{\log^2 n}{V_k} \right). \label{eqn:variance_decomposition}
	\end{align}
	By Lemma~\ref{lem_XVH_not_too_small} all three terms are $O \left( \frac{1}{\eps_k n} \right)$, which proves~\eqref{eqn_variance_estimate_A_bis}. The proofs of~\eqref{eqn_variance_estimate_B} and~\eqref{eqn_variance_estimate_C} from Proposition~\ref{prop_variance_estimates_v1} are the same with simpler terms, so we only have the more accurate estimate~\eqref{eqn_variance_estimate_A} left to prove.
	
	For this, we also start from~\eqref{eqn:variance_decomposition}. Let us first handle the error terms one by one. The first error term $O \left( \frac{Z_k^2}{\eps_k^{3/2}n} \right)$ is $O \left( \frac{\eps_k^{1/2}}{\eps_k n} \right)$ by Lemma~\ref{lem_XVH_not_too_small}. Similarly, by Lemma~\ref{lem_XVH_not_too_small} we have $X_k=O(\eps_k^2 n \times n^{1/100})$ so the second error term $O \left( \frac{X_k}{V_k \eps_k^{3/2} n} \right)$ is $O \left( \frac{\eps_k^{1/2} n^{1/100}}{\eps_k n} \right)$. Finally, also by Lemma~\ref{lem_XVH_not_too_small}, the last term of~\eqref{eqn:variance_decomposition} is $O \left( \frac{\log^2 n}{\eps_k n} \times \frac{1}{\eps_k n} \right)$, which is $O \left( \frac{n^{-1/30}}{\eps_k n} \right)$ because $\eps_k \geq n^{-2/5-1/100}$.
	
	We now move on to the main term $\frac{Z_k}{\eps_k^{3/2} n}$: by Lemma~\ref{lem:z_smooth}, we have $Z_k=3\frac{S_k}{V_k}+O(Z_k^2)$, and the $O (Z_k^2)$ part gets absorbed in the first error term of~\eqref{eqn:variance_decomposition}. Moreover, using~\eqref{eqn:asymptotic_fluid_excess} we can write
	\[ S_k=\ss \left( \frac{k}{n} \right) n + \eps_k^{1/2} \sqrt{n} \Ct_k = \frac{4\mathrm{e}^{1/2}}{3} \eps_k^{3/2} n \left( 1+O \left( \eps_k^{1/2} \right) + O \left( \eps_k^{-1} n^{-1/2} |\Ct_k| \right) \right) \]
	and similarly using~\eqref{eqn:asymptotic_fluid_vertices} 
	\[ V_k=\sv \left( \frac{k}{n} \right) n + \sqrt{n} \Bt_k = 2  \, \eps_k n \left( 1+O \left( \eps_k \right) + O \left( \eps_k^{-1} n^{-1/2} |\Bt_k| \right) \right). \]
	Since $\eps_k \geq n^{-2/5-1/100}$ and $\Bt_k, \Ct_k = O \left( \log n \right)$, we have \[ \eps_k^{-1}n^{-1/2} |\Bt_k| \leq n^{-1/30} \quad \mbox{and} \quad \eps_k^{-1}n^{-1/2} |\Ct_k| \leq n^{-1/30}.\] It follows that
	\[ \frac{S_k}{V_k} = \frac{2}{3} \mathrm{e}^{1/2} \eps_k^{1/2} \left( 1+O \left( \eps_k^{1/2} \right) + O \left( n^{-1/30} \right) \right). \]
	Therefore, we get
	\[ \frac{Z_k}{\eps_k^{3/2} n} = \frac{2\sqrt{ \mathrm{e}}}{\eps_k n} + O \left( \frac{\eps_k^{1/2}}{\eps_k n} \right) + O \left( \frac{n^{-1/30}}{\eps_k n} \right) + O \left( \frac{Z_k^2}{\eps_k^{3/2} n} \right) \]
	and~\eqref{eqn_variance_estimate_A} follows.
\end{proof}

\section{Good region and a stochastic differential equation}\label{sec:good_region}

\subsection{Rough behaviour of the fluctuations}

The goal of this section is to control the fluctuations of the Markov chain $(X, V, S)$ around its fluid limit by using the drift and variance estimates of Section~\ref{sec:drift_variance}. Since most of the proofs of this section are straightforward adaptations from \cite[Section 4.2]{BCC22}, we focus primarily on emphasizing the differences. We keep the same notation as in Section~\ref{sec:drift_variance}. In particular, we recall that $\theta^n$ is the first time where $X$ hits $0$ and that $\widetilde{\theta}^n= \theta^n \wedge \left( t^* n-n^{3/5-1/100} \right)$, with $t^*=1-\frac{3}{2\e}$. We also recall that we are working with the multigraph model $\GG(n,m_n)$ with $m_n=\frac{\e}{2}n+O(\sqrt{n})$.

We start with the control of the rescaled fluctuations $ \widetilde{A}$, $ \widetilde{B}$ and $ \widetilde{C}$ in the ``bulk" of the interval $[0, t^*]$.
For $\eta>0$, we write $ k_0(\eta) = \lfloor (t^* - \eta) n\rfloor$. For $\eta, K>0$, we denote by $\Dd^n_{\eta,K}$ the set of triples $(a,b,c)$ such that $|a|, |b|, |c| \leq K$ and 
\[ \P \left( \At^n_{k_0(\eta)}=a, \Bt^n_{k_0(\eta)}=b, \Ct^n_{k_0(\eta)}=c \right)>0. \]
We also define the event $\mathcal{E}_{\eta, K}^n$ by
\[\mathcal{E}_{\eta, K}^n = \left\{ \widetilde{\theta}^n> k_0(\eta)\ \mbox{ and }\ \left( \At^n_{k_0(\eta)}, \Bt^n_{k_0(\eta)}, \Ct^n_{k_0(\eta)}\right) \in \Dd^n_{\eta,K} \right\}.\]
We will make the dependence in $n$ implicit when there is no ambiguity.

\begin{lemma}[Fluctuations in the bulk]\label{lem:fluctu} For all $\eta, \delta >0$, there exists $K_{\eta, \delta}>0 $ such that for $n$ large enough, we have 
	\begin{equation}\label{eq:flucbulk}
		\P \left( \mathcal{E}^n_{\eta, K_{\eta, \delta}} \right) \geq 1-\delta.
	\end{equation}
\end{lemma}

\begin{proof} The proof of this result is similar to that of~\cite[Lemma 2]{BCC22}, with the difference that here, the initial fluctuations at time $k=0$ are random and of order $\sqrt{n}$ instead of just being $0$. More precisely, by Lemma~\ref{lem:convergence_initial_conditions} and the choice~\eqref{eqn:initial_conditions_fluid_limit} of the initial condition $\left( \sx(0), \sv(0), \ss(0) \right)$, the initial fluctuations $\left( \At_0^n, \Bt_0^n, \Ct_0^n \right)$ are tight.
	
On the other hand, just like in~\cite{BCC22}, we can apply the results of Ethier and Kurtz \cite[Theorem 2.3]{ethier2009markov} \emph{conditionally on the initial fluctuations} $\left( \At_0^n, \Bt_0^n, \Ct_0^n \right)$. More precisely, if $z_n \in \R^3$ are such that $\P \left( \left( \At_0^n, \Bt_0^n, \Ct_0^n \right)=z_n \right)>0$ and $z_n \to z_{\infty}$, then the law of $\left( \widetilde{A}^n_{k_0 ( \eta) }, \widetilde{B}^n_{k_0( \eta)}, \widetilde{C}^n_{k_0( \eta)} \right)$ conditionally on $\left( \At_0^n, \Bt_0^n, \Ct_0^n \right)=z_n$ converges to the law of a random Gaussian vector $G_{\eta}(z_{\infty})$ whose mean and covariances depend on $\eta$ and $z_{\infty}$. That is, for any bounded continuous function $f$, we have
\begin{equation}\label{eqn:eta_cond_on_initial}
	\E \left[ f \left( \widetilde{A}^n_{k_0 ( \eta) }, \widetilde{B}^n_{k_0( \eta)}, \widetilde{C}^n_{k_0( \eta)} \right) | \left( \At_0^n, \Bt_0^n, \Ct_0^n \right)=z_n \right] \xrightarrow[n \to +\infty]{} \E \left[ f(G_{\eta}(z_{\infty})) \right].
\end{equation}

Therefore, let us fix a subsequence. Up to extracting a subsubsequence, we may assume that $\left( \At_0^n, \Bt_0^n, \Ct_0^n \right)$ converges in distribution to a random variable $Z_{\infty}$. By the Skorokhod convergence theorem, let us assume that the convergence is almost sure. Then by~\eqref{eqn:eta_cond_on_initial}, for any bounded continuous $f$, we have
\[ \E \left[ f \left( \widetilde{A}^n_{k_0 ( \eta) }, \widetilde{B}^n_{k_0( \eta)}, \widetilde{C}^n_{k_0( \eta)} \right) | \left( \At_0^n, \Bt_0^n, \Ct_0^n \right) \right] \xrightarrow[n \to +\infty]{a.s.} \E \left[ f(G_{\eta}(Z_{\infty})) \right]. \]
By dominated convergence, this shows that $\left( \widetilde{A}^n_{k_0 ( \eta) }, \widetilde{B}^n_{k_0( \eta)}, \widetilde{C}^n_{k_0( \eta)} \right)$ converges in distribution. We have proved that any subsequence admits a subsubsequence along which the triple $\left( \widetilde{A}^n_{k_0 ( \eta) }, \widetilde{B}^n_{k_0( \eta)}, \widetilde{C}^n_{k_0( \eta)} \right)$ converges in distribution, which shows that it is tight.

Moreover, the same proof applies to show that the variable $\max_{0 \leq k \leq k_0(\eta)} |\At^n_k|$ is tight, which implies $X_k>0$ for $0 \leq k \leq k_0(\eta)$ with probability $1-o(1)$, so $\widetilde{\theta}^n> k_0(\eta)$.
\end{proof}

We will now see the time $k_0(\eta)$ as the starting point of our process ``near the end". In order to come back from the multigraph to the simple graph model in Section~\ref{sec:back}, it will be convenient to prove a version of Theorem~\ref{thm:multi} that holds conditionally on $(X,V,S)_{k_0(\eta)}$. Moreover, although such a conditional result was not written as such in~\cite{BCC22}, its proof for the model studied in~\cite{BCC22} would be exactly the same. We recall that $\D_i(n)$ stands for the number of vertices of degree $i$ in the Karp--Sipser core of $\mathbb{G}(n,m_n)$.

\begin{theorem}\label{thm:bis} There is $\eta_0>0$ such that the following holds. Let $0<\eta<\eta_0$ and $K>0$, and let $(a_n, b_n, c_n)$ be a sequence such that $(a_n, b_n, c_n) \in \Dd^n_{\eta,K}$ for all $n$. Then we have the following convergence for conditional distributions:
\begin{equation}\label{eqn:unif_cond_convergence_degrees}
\left. \left(\begin{array}{c}n^{-3/5} \cdot \D_{2}(n) \\  n^{-2/5} \cdot \D_{3}(n) \\ n^{-1/5} \cdot \D_4(n) \\ \D_5(n) \\ \sum_{i \geq 6} \D_i (n)\end{array}\right) \right| \left\{ \left( \At^n, \Bt^n, \Ct^n \right)_{k_0(\eta)} = (a_n,b_n,c_n) \right\} \xrightarrow[n \to +\infty]{(d)}
\left( \begin{array}{c} \frac{2^{9/5}3^{4/5}}{\mathrm{e}^{3/5}} \vartheta^{-2} \\  
	\frac{2^{16/5}3^{1/5}}{\mathrm{e}^{2/5}} \vartheta^{-3} \\ \frac{2^{13/5}3^{3/5}}{\mathrm{e}^{1/5}} \vartheta^{-4} \\ \mathrm{Poi} \left( \frac{48}{5} \vartheta^{-5}  \right) \\0 \end{array}\right),
\end{equation}
where $ \vartheta=\inf\{t \geq 0 : W_t=t^{-2}\}$ for a standard Brownian motion $W$ started from $0$.
\end{theorem}

We first check that this result easily implies Theorem~\ref{thm:multi}.

\begin{proof}[Proof of Theorem~\ref{thm:multi} given Theorem~\ref{thm:bis}]
	The proof just consists of combining Lemma~\ref{lem:fluctu} with Theorem~\ref{thm:bis}. More precisely, let us fix $\eta>0$ small enough so that the conclusion of Theorem~\ref{thm:bis} is true. Let $g:\R^5 \to \R$ be continuous and bounded, and let $\delta>0$. Finally, let $K=K_{\eta, \delta}$ be given by Lemma~\ref{lem:fluctu}. For the sake of brevity, we denote by $\mathbf{D}(n)$ and $\mathbf{D}(\infty)$ the two $5$-dimensional vectors appearing in~\eqref{eqn:unif_cond_convergence_degrees}. Then we have
	\begin{multline}
		\left| \E \left[ g(\mathbf{D}(n)) \right] - \E \left[ g(\mathbf{D}(\infty)) \right] \right| \\
		\leq 2 \| g \|_{\infty} \P \left( \mathcal{E}_{\eta, K}^c \right) + \E \left[ \mathbbm{1}_{\mathcal{E}_{\eta, K}} \left| \E \left[ g(\mathbf{D}(n)) | \Ff_{k_0(\eta)} \right] - \E \left[ g(\mathbf{D}(\infty)) \right] \right| \right].\label{eqn:E_or_not_E}
	\end{multline}
	For $n$ large enough, the first term is bounded by $2 \delta \| g \|_{\infty}$ by Lemma~\ref{lem:fluctu}. On the other hand, by the Markov property, we have
	\[ \E \left[ g(\mathbf{D}(n)) | \Ff_{k_0(\eta)} \right] = \E \left[ g(\mathbf{D}(n)) \Big| \left(\At, \Bt, \Ct \right)_{k_0(\eta)} \right], \]
	and Theorem~\ref{thm:bis} shows that this converges to $\mathbb{E}[g(\mathbf{D}(\infty))]$, uniformly in the values $ \left(\At, \Bt, \Ct \right)_{k_0(\eta)}$, provided $\mathcal{E}_{\eta,K}$ occurs.
	In particular, the second term of~\eqref{eqn:E_or_not_E} is bounded by $\delta$ for $n$ large enough, which proves all the convergences of Theorem~\ref{thm:multi}. Finally, the claim that the core is a configuration model is an immediate consequence of Lemmas~\ref{lem:markov} and~\ref{lem:config_model}.
\end{proof}

In order to prove Theorem~\ref{thm:bis}, we will apply the same strategy as in~\cite{BCC22} and make sure for each intermediate result that the conditioning on $\Ff_{k_0(\eta)}$ is not a problem. The first step is to obtain rough upper bounds for the fluctuations $\widetilde{A}_k$, $\widetilde{B}_k$ and $ \widetilde{C}_k$. We will rely on the drift and variance estimates from Section~\ref{sec:drift_variance}. This is the analog of~\cite[Proposition 6]{BCC22}.

\begin{proposition}[Rough upper bounds]\label{prop:roughbounds} There is $\eta_0>0$ such that for all $0<\eta<\eta_0$ and $\delta >0$ and $K>0$, there exists a constant $K'$ such that for $n$ large enough, under the event $\mathcal{E}_{\eta,K}$, we have 
	\begin{equation}\label{eqn:defn_good_region}
		\P \left( \max_{k_0(\eta) \leq k < \widetilde{\theta}^n } \left\{ \frac{ \widetilde{A}_k}{| \log ( \eps_k)|^{3/4}}, \widetilde{B}_k, \widetilde{C}_k\right\} \leq K' \Big| \Ff_{k_0(\eta)} \right) \geq 1-\delta.
	\end{equation}
\end{proposition}

\begin{proof}[Sketch of Proof] The proof of this proposition is very similar to that of \cite[Proposition 6]{BCC22}. For this reason, we will just recall the main steps. First, rather than $ \widetilde{C}$, we consider instead the process $ \widetilde{E}$ defined by $ \widetilde{E}_k = \widetilde{C}_k  -  \sqrt{ \mathrm{e}} \widetilde{B}_k$ in view of the form of the drift estimate \eqref{eqn_drift_estimate_C}. We thus say that $( \widetilde{A}_k, \widetilde{B}_k, \widetilde{C}_k)$ is in the \emph{good region} as long as $\max \left( \frac{\At_k}{|\log(\eps_k)|^{3/4}}, \Bt_k, \widetilde{E}_k \right) \leq K'$. We denote by $ L$ the first time at which the process exits this good region and evaluate separately the probability to exit this good region via one of the three processes $ \widetilde{A}, \widetilde{B}, \widetilde{E}$. In each case, we will decompose the process into its predictable and martingale parts.	For all the proof, we fix $\eta>0$, whose value will be precised later, and work under the event $\mathcal{E}_{\eta,K}$.\\
	\textsc{Control of $ \widetilde{B}$.} We start from $ k_0(\eta) = \lfloor (t^* - \eta) n\rfloor$. We can assume $K'>K$, so the event $\mathcal{E}_{\eta,K}$ implies that $L>k_0(\eta)$, so we can apply the estimates of Propositions~\ref{prop:drift_estimates} and~\ref{prop:variance_estimates} until the exit time $L$ of the good region. We then  decompose the fluctuations along scales of the form $ k_j := \lfloor (t^* - 2^{-j}\eta) n\rfloor$. Note that the estimate \eqref{eqn_drift_estimate_B} is even better than its equivalent \cite[Equation~(16)]{BCC22}. By the same computation as for~\cite[Proposition 6]{BCC22}, we obtain that	for all $k_0(\eta) \leq k \leq \widetilde{\theta}^n$, we have
	\begin{eqnarray*}
		\mathbb{1}_{k \leq L} \left|  \sum_{ \ell = k_0(\eta)}^{k-1} \mathbb{E} \left[ \Delta \widetilde{B}_{\ell}| \mathcal{F}_{ \ell} \right ]\right| \leq K' \cdot ( \mathrm{Cst} \cdot \sqrt{ \eta} | \log (\eta)|),
	\end{eqnarray*}
	where $ \mathrm{Cst}>0$ is an absolute constant. In particular, using Doob's maximal inequality, our estimate \eqref{eqn_variance_estimate_B} in the good region and the assumption that $|\Bt_{k_0(\eta)}| \leq K<K'/2$ if $K'$ was chosen large enough, we obtain
	\begin{equation}\label{eqn:exit_Doob_B}
		\mathbb{P} \left( L < \widetilde{\theta}^n \mbox{ and we exit the good region by } \widetilde{B} | \Ff_{k_0(\eta)} \right) \leq \frac{4}{ K'^2 /4} \E \left[ (M_L^B)^2 | \Ff_{k_0(\eta)} \right],
	\end{equation}
	where $M^B$ is the martingale part of $\Bt$. Note that this estimate is slightly simpler to obtain than in~\cite{BCC22} since we condition on $\Ff_{k_0(\eta)}$, so we do not need to handle the term $\P \left(  |\Bt_{k_0(\eta)}| > K' \right)$. By orthogonality of martingale increments in $L^2$, Equation~\eqref{eqn:exit_Doob_B} becomes
	\begin{align*}
	  \mathbb{P} \left( L < \widetilde{\theta}^n \mbox{ and we exit the good region by } \widetilde{B} | \Ff_{k_0(\eta)} \right) &\leq \frac{16}{ K'^2} \sum_{k = k_{0}}^{ \infty} \E \left[ \mathbb{1}_{k \leq L} \mathrm{Var} ( \Delta \widetilde{B}_k | \mathcal{F}_{k}) | \Ff_{k_0(\eta)} \right]\\ &\leq \frac{16 K'' \eta}{K'^2}, \nonumber
	\end{align*}
	where $K''$ is an absolute constant coming from the estimate \eqref{eqn_variance_estimate_B}. This can be made smaller than $\delta$ by taking $ K'$ large enough. \\
	\textsc{Control of $ \widetilde{E}$.} To prove that the process does not exit the good region via $ \widetilde{E}$, we use more carefully the ``restoring force" effect of its conditional drift. In particular, we introduce $ L_{ E}^-$ the last time between $k_0(\eta)$ and $ L$ at which $\widetilde{E}$ is smaller than $ K'/2$. For all $L_E^- \leq k \leq L$, using the form of the expected drift given by \eqref{eqn_drift_estimate_B} and \eqref{eqn_drift_estimate_C}, we can prove that, if $\eta$ was chosen small enough, we have for all $L_E^- < k < L$, 
	\begin{eqnarray*} \mathbb{E} \left[ \Delta \widetilde{E}_k | \mathcal{F}_k\right ] \leq - \frac{K'}{ 4 \eps_k n}<0,\end{eqnarray*}
	which ``pulls back" the fluctuation $ \widetilde{E}$ towards $0$. Thus, on the event $\{ k_0(\eta) < L_E^-<L<\widetilde{\theta}^n \mbox{ and } \widetilde{E}_L > K'\}$, the variation of the martingale part of $ \widetilde{E}$ must be larger than $ K'/2$ over $[ L_E^-,L]$ to compensate the predictable part. Using Doob's maximal inequality and our variance estimates \eqref{eqn_variance_estimate_B} and \eqref{eqn_variance_estimate_C}, we can show that the probability that the martingale part compensates the predictable part over $[ L_E^-,L]$ is smaller than $ \mathrm{Cst} \sqrt{ \eta}/ K'^2$. If $K'$ was chosen large enough, we thus obtain 
	$$ \mathbb{P} \left( k_0(\eta) < L_E^-<L<T \mbox{ and } \widetilde{E}_L > K'| \Ff_{k_0(\eta)} \right) \leq \delta.$$
	Combining this with the symmetric case $\widetilde{E}_L <- K'$, this finishes the control of $ \widetilde{E}$. \\
	\textsc{Control of $ \widetilde{A}$.} The control of $ \widetilde{A}$ is the most subtle since we use the strength of the restoring force of the drift (and not only its sign). The drift estimate \eqref{eqn_drift_estimate_A} has the same form as \cite[Equation~(15)]{BCC22} in the good region. Moreover, the variance estimate \eqref{eqn_variance_estimate_A_bis} only differs by a (multiplicative) constant from \cite[Equation~(19)]{BCC22}, which is not important for this rough control. Thus the same proof shows that the probability to exit the good region via $ \widetilde{A}$ is small and concludes the proof.
	
	We highlight that in all the proof, the reason why we needed to take $\eta$ small is that the drift estimate of~\eqref{eqn_drift_estimate_A} (resp.~\eqref{eqn_drift_estimate_C}) remains of order $-\frac{\At_k}{\eps_k n}$ (resp. $-\frac{\widetilde{E}_k}{\eps_k n}$). Since the constants in Proposition~\ref{prop:drift_estimates} are absolute, the bound $\eta_0$ on how small $\eta$ needs to be depends neither on $\delta$ nor on $K$. This will be important later to avoid circular dependencies (in Section \ref{sec:back}, we will choose $\eta$ depending on $\delta$).
\end{proof}

\subsection{Reaching the end via a stochastic differential equation}\label{subsec:SDE}

We recall that the number of leaves after $k$ steps is $X_k=n\sx \left( \frac{k}{n} \right)+\eps_k^{3/4} \sqrt{n} \At_k$. By Proposition \ref{prop:roughbounds} and~\eqref{eqn:asymptotic_end_x}, we deduce that the process $X_k$ stays positive as long as 
$$ t^* n - k \gg n^{3/5} (\log n)^{3/5}.$$
We will now refine the control on $ \widetilde{A}$ in order to remove the $\log^{3/5} n$ factor. To look more precisely at times of order $n^{3/5}$ before the end, we introduce, for all $k \geq 0$, the notation 
$$ t_k := n^{-3/5} ( t^* n -k) \quad \mbox{so that} \quad k = t^* n - t_k n^{3/5} \mbox{ and } \eps_k = t_k n^{-2/5}.$$

\begin{proposition}[Control of $ \widetilde{A}$ near the end]\label{prop:controlA} There is $\eta_0>0$ such that for all $0<\eta<\eta_0$ and $ \delta >0$ and $K>0$, there exists $ K'$ such that for $n$ large enough, under the event $\mathcal{E}_{\eta,K}$, we have
\[ \P \left( \mbox{$\forall k \geq k_0(\eta)$ such that $t_k \geq K'$, we have $| \At_k| < K' t_k^{1/8}$}  \, | \, \Ff_{k_0(\eta)} \right) \geq 1-\delta.\]
\end{proposition}

The proof is essentially the same as that of \cite[Proposition 7]{BCC22}, the only differences being the conditioning on $\Ff_{k_0(\eta)}$ (which does not change anything to the argument), and the value of the constant in front of the variance (which does not play an important role). For this reason, we do not reproduce it here. 
Equipped with this control on the fluctuations, we now have all the tools to study the convergence of the stopping time $\theta^n$. We note that $t_{ \theta^n} \in \mathbb{R}$ is given by
\begin{eqnarray*} \theta^n = t^* n - t_{ \theta^n} n^{3/5}.\end{eqnarray*}
We now show the convergence in distribution of $t_{ \theta^n}$ to a certain random variable. We will need a statement in the same form as for Theorem~\ref{thm:bis}, i.e. a uniform result conditionally on $\Ff_{k_0(\eta)}$.

\begin{proposition}[Convergence of the stopping time]\label{prop:conv_stopping_time}
	There is $\eta_0>0$ such that the following holds. Let $0<\eta<\eta_0$ and $K>0$, and let $(a_n, b_n, c_n) \in \Dd^n_{\eta,K}$ for all $n$. Then we have the convergence of conditional distributions
	\begin{equation}\label{eqn:conditioning_Fkzero}
	t_{\theta^n} \Big| \left\{ \left( \At^n, \Bt^n, \Ct^n \right)_{k_0(\eta)}=(a_n,b_n,c_n) \right\} \xrightarrow[n \to +\infty]{(d)} 6^{4/5} \mathrm{e}^{-3/5} \vartheta^{-2},
	\end{equation}
	where $ \vartheta := \inf \{ t \geq 0 : W_t = -t^{-2} \}$ is the hitting time of the curve $ t \mapsto -t^{-2}$ by a standard Brownian motion $W$ started from $0$ at time $0$.
\end{proposition}

\begin{proof}

Let $\eta_0$ be given by Propositions~\ref{prop:roughbounds} and~\ref{prop:controlA}, and let us fix $0<\eta<\eta_0$, $K>0$ and $(a_n,b_n,c_n) \in \Dd^n_{\eta,K}$ for all $n$. In all the proof, we work under the conditioning of~\eqref{eqn:conditioning_Fkzero}. The form of the drift and variance estimates in Equations \eqref{eqn_drift_estimate_A} and \eqref{eqn_variance_estimate_A} suggests that near the end, the fluctuation process $( \widetilde{A}_{t^* n - t n^{3/5}} : - \infty < t < t_{ \theta^n})$ converges weakly (for the $ \|\ \|_{ \infty}$-norm) toward a process $\left( \mathcal{A}_{t} \right)_{-\infty < t \leq \tau}$ satisfiying 
\begin{equation*} \mathrm{d} \mathcal{A}_t = - \frac{ \mathcal{A}_t}{4|t|} \mathrm{d}t + \frac{2 \sqrt{ \mathrm{e}}}{t} \mathrm{d} W_t,
\end{equation*}
where $W$ is a standard Brownian motion. To make this heuristic precise, the proof is similar to that of \cite[Proposition 8]{BCC22}. In particular, we need to consider the renormalized process 
\begin{equation*} \widetilde{F}_k = \frac{ \widetilde{A}_k}{ t_k^{1/4}}, \quad 0 \leq k \leq \theta^n. 
\end{equation*}
We fix $ \delta >0$ and we take $K' > 0$ such that on an event of probability at least $ 1 - 2 \delta$, the conclusions of Proposition~\ref{prop:roughbounds} and~\ref{prop:controlA} hold. We also fix $\xi \in \left( 0, K'^{-1} \right)$  such that $ K' \xi^{1/8} \leq \delta$. To avoid stopping time issues, we extend $ \widetilde{F}$ after time $ \theta^n$ by a process $ \widehat{F}$ with increments $ \pm \frac{2 \sqrt{ \mathrm{e}}}{ t_k^{3/2} n^{3/5}}$ with probability $1/2$, as in the proof of \cite[Proposition 8]{BCC22}. Then with probability at least $ 1 - 2 \delta$,  for all $k$ such that $ \xi < t_k < \xi^{-1}$, we have 
$$ \left\{\begin{array}{l}
\displaystyle| \widehat{F}_{ nt^*- \xi^{{-1}}n^{3/5}}| < \delta, \quad (\mbox{by Prop. \ref{prop:controlA} and the assumption }K' \xi^{1/8} \leq  \delta),
\\
\ \\ 
\displaystyle\mathbb{E} [ \Delta \widehat{F}_k | \mathcal{F}_k] = o(n^{-3/5}) \cdot | \widehat{F}_k| +o(n^{-3/5}), \\
\displaystyle \mathrm{Var} \left( \Delta \widehat{F}_k | \mathcal{F}_k\right) = \frac{2 \sqrt{ \mathrm{e}}}{ t_k^{3/2} n^{3/5}} + o(n^{-3/5}),\\
\displaystyle \|\Delta \widehat{F}_k\|_{\infty} = o(1).
\end{array}\right.$$
The proof of these estimates relying on Propositions~\ref{prop:drift_estimates} and~\ref{prop:variance_estimates} is exactly the same as in~\cite[Proposition 8]{BCC22}, except for the last one. Indeed, in~\cite{BCC22} the vertex degrees were bounded, so it was immediate that the increments of $X$ are $O(1)$, so the increments of $\widehat{F}$ are $O \left( n^{-3/5} \right)$. In our setting the vertex degrees are not bounded. However, an immediate consequence of~\eqref{eqn:degree_distribution_light_tail} applied to $k=k_0(\eta)$ is that with high conditional probability (under the conditioning of~\eqref{eqn:conditioning_Fkzero}), all the vertex degrees at time $k_0(\eta)$ are smaller than $\log n$. This implies that the increments of $X$ after time $k_0(\eta)$ are at most $\log^2 n$, so $\|\Delta \widehat{F}_k\|_{\infty} = O \left( \frac{\log^2 n}{n^{3/5}} \right)$ on a very high probability event.

Using standard diffusion approximation results (see e.g.~\cite{kushner1974weak}) and the Dubins--Schwarz theorem, and letting $ (\delta, \xi) \to (0,0)$, we deduce the following weak convergence over all compact subsets of $ (0,\infty)$:
\begin{eqnarray}\left(\widehat{F}_{t^* n - tn^{3/5}}\right)_{ 0 < t < \infty } \xrightarrow[n\to\infty]{} 2 \cdot \mathrm{e}^{1/4} \left(W_{ \frac{1}{ \sqrt{t}}}\right)_{ 0 < t < \infty },   \label{eq:convBrownien1}\end{eqnarray}
still under the conditioning~\eqref{eqn:conditioning_Fkzero}, and where $W$ is a standard Brownian motion with $ W_0 = 0$. Then we obtain
$$  t_{ \theta^n} \xrightarrow[n\to\infty]{(d)} \tau := \sup \{ t \geq 0 | 2 \cdot \mathrm{e}^{1/4} W_{ \frac{1}{ \sqrt{t}}} =  - \frac{\mathrm{e}}{3} t\}, $$
still under the same conditioning. By scaling, we finally find $\tau=\frac{6^{4/5}}{ \mathrm{e}^{3/5}}\vartheta^{-2}$, where $\vartheta$ is given by
\begin{equation}\label{eqn:defn_vartheta}
	\vartheta =  \inf \{ t \geq 0 | W_{ t} =  t^{-2}  \}.
\end{equation}
\end{proof}

\subsection{Size and composition of the Karp--Sipser core}

We have now all the tools to finish the proof of Theorem~\ref{thm:bis}. We first recall that by Lemma~\ref{lem:markov}, conditionally on $\left( V_{\theta^n}, S_{\theta^n} \right)$, the core follows the distribution $\mathbb{G} \left(0,V_{\theta^n}, S_{\theta^n} \right)$. Therefore, we first need to estimate $\left( V_{\theta^n}, S_{\theta^n} \right)$.

For this, let us work under the same conditioning as in~\eqref{eqn:conditioning_Fkzero}. By Proposition \ref{prop:conv_stopping_time}, the renormalized stopping time $ t_{\theta^n}$ converges in distribution towards $6^{4/5} \mathrm{e}^{-3/5} \vartheta^{-2}$. Hence, by definition of the fluctuations $B$, we can write
$$ V_{\theta^{n}} \quad = \underbrace{B_{\theta^{n}}}_{  \underset{ \mathrm{Prop}. \ref{prop:roughbounds}}{=} O \left( \sqrt{n} \log(n)^{3/4} \right)} + \underbrace{n \mathscr{V} \left( \frac{\theta^{n}}{n}\right)}_{ \underset{\eqref{eqn:asymptotic_fluid_vertices}}{\sim} 2 t_{\theta^{n}} n^{3/5}} \quad = 2 t_{\theta^{n}} n^{3/5} + o \left( n^{3/5} \right)$$
in probability. Similarly, at time $\theta_n$, the surplus is given by
\begin{eqnarray*}
	S_{\theta^{n}} &=& \underbrace{n \mathscr{S} \left( \frac{\theta^{n}}{n}\right)}_{ \underset{\eqref{eqn:asymptotic_fluid_excess}}{\sim} \frac{4 \sqrt{ \mathrm{e}}}{3} t_{\theta^n}^{3/2} n^{2/5}} + \underbrace{C_{\theta^n}}_{\underset{ \mathrm{Prop}. \ref{prop:roughbounds}}{=}  O \left( n^{3/10} \log(n) \right)}, \\
	&=& \frac{4 \sqrt{ \mathrm{e}}}{3} t_{\theta^n}^{3/2} n^{2/5} + o \left( n^{2/5} \right)
\end{eqnarray*}
in probability. In other words, still under the conditioning of~\eqref{eqn:conditioning_Fkzero}, we have the convergence in distribution
\begin{align} \left( \frac{V_{\theta^n}}{n^{3/5}}, \frac{S_{\theta^n}}{(V_{\theta^n})^{2/3}}\right)\xrightarrow[n\to\infty]{(d)} &\left( 2 \cdot 6^{4/5} \mathrm{e}^{-3/5} \vartheta^{-2}, \frac{4 \sqrt{ \mathrm{e}}(6^{4/5} \mathrm{e}^{-3/5} \vartheta^{-2})^{3/2}}{3(2 \cdot 6^{4/5} \mathrm{e}^{-3/5} \vartheta^{-2})^{2/3}}\right)\nonumber  \\
=& \left( 2^{9/5} 3^{4/5} \mathrm{e}^{-3/5}\vartheta^{-2}, 2 \cdot 3^{-1/3} \vartheta^{-5/3} \right).\label{eq:conv_end} \end{align}
To pass from the two parameters $V_{\theta^n}, S_{\theta^n}$ to a full understanding of the vertex degrees, we will use the following lemma.

\begin{lemma}\label{lem:compo} Let $x>0$ and let $(v_n),(s_n)$ be sequences of nonnegative integers such that $v_n \to +\infty$ and $s_n = x v_n^{2/3}+ o (v_n^{2/3})$ as $n \to +\infty$. We also assume that $\log^4 n \leq v_n \leq n$ and that $s_n \leq n$.
	Then we have the following convergences in distribution:
	\begin{align*}
	\frac{\# \{ \mbox{vertices of degree $ 2$ in $\mathbb{G}(0, v_n, s_n)$}\}}{v_n} &\xrightarrow[n\to\infty]{(d)} 1,\\
	\frac{\# \{ \mbox{vertices of degree $ 3$ in $\mathbb{G}(0, v_n, s_n)$}\}}{x v_n^{2/3}} &\xrightarrow[n\to\infty]{(d)} 1,\\
	\frac{\# \{ \mbox{vertices of degree $ 4$ in $\mathbb{G}(0, v_n, s_n)$}\}}{(3/4) \cdot x^2 v_n^{1/3}} &\xrightarrow[n\to\infty]{(d)} 1,\\
	\# \{ \mbox{vertices of degree $ 5$ in $\mathbb{G}(0, v_n, s_n)$}\} &\xrightarrow[n\to\infty]{(d)} \mathrm{Poisson}\left( \frac{9 x^3}{20}\right),\\
	\# \{ \mbox{vertices of degree at least $6$ in $\mathbb{G}(0, v_n, s_n)$}\} &\xrightarrow[n\to\infty]{(d)} 0.
	\end{align*}
\end{lemma}

From here, we can easily deduce Theorem~\ref{thm:bis}.

\begin{proof}[Proof of Theorem~\ref{thm:bis}]
Let $\eta_0$ be given by Proposition~\ref{prop:conv_stopping_time}. As before, we fix $0<\eta<\eta_0$, $K>0$ and $(a_n,b_n,c_n) \in \Dd^n_{\eta,K}$ for all $n$. As it is the only case with two levels of randomness involved, we only write down the proof completely for vertices of degree $5$. We recall that $\D_5(n)$ stands for the number of vertices of degree $5$ in the Karp--Sipser core of $\mathbb{G}(n,m_n)$. We will also write $\mathbb{E}_{a,b,c}[\cdot]$ for conditional expectations on the event $(\At, \Bt, \Ct)_{k_0(\eta)}=(a_n,b_n,c_n)$.
	
Let $g,h : \mathbb{R} \mapsto \mathbb{R}$ be two bounded continuous functions. By conditioning on $\left( V_{\theta^n}, S_{\theta^n} \right)$ and using Lemma~\ref{lem:markov}, we can write
	\begin{equation}\label{eqn:conditioning_on_s_and_v}
		\mathbb{E}_{a,b,c} \left[ g \left( \frac{V_{\theta^n}}{n^{3/5}}\right) h \left( \D_5(n) \right) \right] = \mathbb{E}_{a,b,c} \left[ g \left( \frac{V_{\theta^n}}{n^{3/5}}\right) H \left( V_{\theta^n}, \frac{S_{\theta^n}}{V_{\theta^n}^{2/3}} \right) \right],
	\end{equation}
	where
	\[ H \left( v,x \right) = \E \left[ h \left( \# \{ \mbox{vertices of degree $5$ in $\mathbb{G} (0,v, x v^{2/3} )$} \} \right) \right].  \]
	By the Skorokhod embedding theorem, we may assume that the convergence in distribution~\eqref{eq:conv_end} is almost sure. Moreover, Lemma~\ref{lem:compo} implies that when $v_n \to +\infty$ and $x_n \to x$, the quantity $H \left( v_n, x_n \right)$ converges to $\E \left[ h \left( \mathrm{Poisson}(9x^3/20) \right) \right]$. Therefore, we have almost sure convergence of the quantity inside the expectation in the right-hand side of~\eqref{eqn:conditioning_on_s_and_v}. By dominated convergence, we finally obtain
	\[ \mathbb{E}_{a,b,c} \left[ g \left( \frac{V_{\theta^n}}{n^{3/5}}\right) h \left( \D_5(n) \right) \right] \xrightarrow[n \to +\infty]{} \mathbb{E}_{a,b,c} \left[ g \left( \frac{2^{9/5} 3^{4/5}}{ \mathrm{e}^{3/5}\vartheta^{2}} \right) \times \mathbb{E} \left[ h\left( \mathrm{Poisson} \left( \frac{9}{20} \left(\frac{2}{ 3^{1/3} \vartheta^{5/3}}\right)^3\right)\right)\Big| \vartheta  \right ] \right]. \]
	In other words, the pair $\left( \frac{V_{\theta^n}}{n^{3/5}}, \D_5(n) \right)$ under $\mathbb{P}_{a,b,c}$ converges to a pair of random variables where the first coordinate has the law of $2^{9/5} 3^{4/5} \mathrm{e}^{-3/5} \vartheta^{-2}$ and, conditionally on the first, the second follows a Poisson distribution. This proves the convergence of $\D_5(n)$.
	
	Finally, the argument for vertices of other degrees is the same, except that we need to normalize $\D_i(n)$ by the right power of $V_{\theta^n}$ before applying $h$. Note that it is immediate that all the convergences in distribution of Theorem~\ref{thm:bis} are joint, as all convergences except one are actually convergences in probability to a constant.
\end{proof}

Finally, we can now prove Lemma~\ref{lem:compo}, which is a consequence of the Poisson approximation for the vertex degrees provided by Lemma~\ref{lem:degrees_are_Poisson}.

\begin{proof}[Proof of Lemma \ref{lem:compo}]The proof consists of a first and second moment computation which makes use of Lemma~\ref{lem:degrees_are_Poisson}. For the sake of brevity, in this proof, we will write $\mathbb{G}(n)$ for $\mathbb{G}(0,v_n,s_n)$. We also write $z_n=z(v_n, s_n)$. By Lemma~\ref{lem:z_smooth}, our assumptions imply $z_n = 3 x v_n^{-1/3} + o(v_n^{-1/3})$ as $n \to +\infty$. In particular, this implies that $v_n$, $s_n$ and $z_n$ satisfy the assumptions of Lemma~\ref{lem:degrees_are_Poisson}. We also recall that the $v_n$ vertices of $\mathbb{G}(n)$ are labelled from $1$ to $v_n$. Using~\eqref{eqn:degree_distribution_moderate}, for $2 \leq d \leq 5$, we have
	\begin{align*} 
		\mathbb{E} \left[ \# \{ \mbox{vertices of degree $d$ in $\mathbb{G}(n)$}\} \right] &= v_n \cdot \mathbb{P} \left( \mathrm{deg} (1) = \ell\right) \\
		&= v_n \frac{z_n^d}{d! f(z_n)} \left( 1 + O \left( \frac{1}{v_n \cdot v_n^{-1/3}} \right)\right) \\
		&= (1+o(1)) \frac{2 (3x)^{d - 2}}{d !}\cdot v_n^{ \frac{5 - d}{3}}.
	\end{align*}
	Similarly, using~\eqref{eqn:degree_distribution_moderate} and~\eqref{eqn:degree_distribution_light_tail}, we have
	\begin{align*}
		\mathbb{E} \left[ \# \{ \mbox{vertices of degree $\geq 6$ in $\mathbb{G}(n)$}\} \right] &= O \left( v_n \sum_{d=6}^{\log n} \frac{z_n^d}{d! f(z_n)} \right) +O \left( n^{-9} \right) \\
		&= O \left( v_n z_n^4 \right) + O \left( n^{-9} \right).
	\end{align*}
	In particular this is $o(1)$, which is sufficient to prove the part of the Lemma on vertices of degree at least $6$.
	
	We now move on to the second moment computation. By Equation \eqref{eqn:degree_distribution_two_pt_function}, for all $2 \leq d \leq 4$, we have
	\begin{align*} 
		\mathbb{E} \left[ \# \{ \mbox{vertices of degree $d$ in $G(n)$}\}^2 \right] &= v_n \cdot \mathbb{P} \left( \mathrm{deg} (1) = d \right)  + v_n (v_n-1)  \mathbb{P} \left( \mathrm{deg} (1) = \mathrm{deg} (2) = d \right) \\
		&= O(v_n^{ \frac{5-d}{3}}) + \left( v_n^2 + O(v_n) \right) \frac{z_n^{2 d}}{d!^2 f(z_n)^2} \left( 1+ O \left( \frac{\log^2 n}{v_n^{2/3}} \right) \right) \\
		&= (1 + o(1)) \frac{4 (3x)^{2d-4}}{d!^2} v_n^{\frac{10 - 2d}{3}} . \\
		&= (1 + o(1)) \mathbb{E} \left[ \# \{ \mbox{vertices of degree $ \ell$ in $G(n)$}\} \right]^2,
	\end{align*}
	where we have used the assumption $v_n \geq \log^4 n$ to get rid of the error term. This is sufficient to prove the convergence of the number of vertices of degree $d$ for $d \in \{ 2,3,4\}$. However, for $d=5$, we cannot expect concentration around the expectation, so we will need to compute all the moments to prove convergence in distribution.
	
	More precisely, we fix $m \geq 1$.	For $x \in \mathbb{R}$, we denote by $(x)_{m}:= x (x-1) \dots (x-m+1)$ the falling factorial of $x$. Using Equation~\eqref{eqn:m_pt_function_degree_five} from Lemma~\ref{lem:degrees_are_Poisson}, we have
	\begin{align*} 
		\mathbb{E} \left[ \left(\# \{ \mbox{vertices of degree $  5 $ in $\mathbb{G}(n)$}\}\right)_{m} \right] &= \sum_{j_1 = 1}^{v_n} \sum_{j_2 \neq j_1}^{}\cdots \sum_{j_m \notin \{j_1, \dots, j_{m-1}\}}^{} \mathbb{P} \left( \mathrm{deg} (j_1) = \dots = \mathrm{deg} (j_m) = 5 \right)\\
		&= (v_n)_m \cdot (1 + o(1)) \left( \frac{2\cdot 3^3 x^3 v_n^{-1}}{5!}  \right)^{m} \\
		&= (1 + o(1)) \left(\frac{9}{20} x^3 \right)^m
	\end{align*}
	as $n \to +\infty$. This shows that the factorial moments of the number of vertices of degree $5$ converge to those of a Poisson variable, which implies convergence in distribution.
\end{proof}

\section{From the configuration model back to Erd\H{o}s--R\'enyi}\label{sec:back}

We now extend our result to simple graphs in order to prove Theorem~\ref{thm:main}. We will rely on the fact that the graph $ \mathrm{G}(n,m)$ has the law of the multigraph $\mathbb{G}(n,m)$, conditioned to be simple. Note that passing from one model to the other is not obvious: since the probability that $\mathbb{G}(n,m)$ is simple does not go to one in the regime we are interested in, the conditioning might ``twist" the limiting distribution. Another natural approach, which is followed in~\cite{glasgow2024central}, would be to use the fact that $ \mathrm{G}(n,m)$ and $\mathbb{G}(n,m)$ can be coupled so that the edit distance between them (i.e. the number of edges to add or remove to pass from one to the other) is $O(1)$. Unfortunately, this does not adapt to our problem, since contrary to the matching number, the size of the Karp--Sipser core is not a Lipschitz function of the graph for the edit distance.

To pass from one model to the other, it will be useful to know the probability that a configuration model is a simple graph, so we recall the following classical result (see e.g.~ \cite[Corollary 5.3]{mckay1991asymptotic}). If $\mathbf{d}=(d_i)_{i \geq 1}$ is a sequence of integers which is $0$ eventually, we write $|\mathbf{d}|=\frac{1}{2}\sum_{i \geq 1} i d_i$. If this quantity is an integer, we denote by $\mathrm{CM}(\mathbf{d})$ a configuration model with $d_i$ vertices of degree $i$ for all $i$.

\begin{lemma}\label{lem:proba_simple}
	Let $\mathbf{d}^{(n)}$ be degree sequences such that $|\mathbf{d}^{(n)}| \to +\infty$ and such that $\max \{ i \geq 1 | d_i^{(n)} >0 \}=o \left( \sqrt{|\mathbf{d}^{(n)}|}\right)$. Then
	\[ \P \left( \mbox{$\mathrm{CM}(\mathbf{d}^{(n)})$ is simple} \right) = \exp \left( -\frac{1}{4} \left( \frac{\sum_{i \geq 1} i^2 d_i^{(n)}}{|\mathbf{d}^{(n)}|} \right)^2 + \frac{1}{4} \right) +o(1) \]
	as $n \to +\infty$.
\end{lemma}

Again, for $\eta>0$, we recall that $ k_0(\eta) = \lfloor ( t^* - \eta) n \rfloor$. We will decompose the event that the graph $ \mathbb{G}( n, m_n)$ is simple as the intersection of three events: 
\begin{itemize}
	\item the event $\simp$ that the execution of the Karp--Sipser algorithm does not encounter any multiple edge or loop until step $k_0(\eta)$,
	\item the event $\simpbis$ that the execution of the Karp--Sipser algorithm does not encounter any multiple edge or loop between step $k_0(\eta)$ and its end $\theta^n$,
	\item the event that the Karp--Sipser core is simple.
\end{itemize}
Note that if a graph contains two multiple edges, then these two edges will be discovered at the same step of the Karp--Sipser algorithm, so the decomposition above is not ambiguous. We want to show that the simplicity of $\mathbb{G}(n,m_n)$ is asymptotically independent from the degrees in its Karp--Sipser core. Roughly speaking, the argument will be to handle the three events of this decomposition as follows:
\begin{itemize}
	\item the event $\simp$ is almost independent from the Karp--Sipser core because Theorem~\ref{thm:bis} holds conditionally of $\Ff_{k_0(\eta)}$ (this is why Theorem~\ref{thm:multi} is not sufficient and we needed Theorem~\ref{thm:bis}),
	\item the event $\simpbis$ has probability close to $1$,
	\item by Lemma~\ref{lem:proba_simple}, the probability that the Karp--Sipser core is simple conditionally on its vertex degrees is almost deterministic.
\end{itemize} 
To make this sketch precise, we start with the second item.

\begin{lemma}\label{lem:intermediate_simplicity}
	Let $\delta>0$. Then there is $\eta>0$ such that for $n$ large enough, we have
	\[ \P \left( \simpbis \right) \geq 1-\delta. \]
\end{lemma}

\begin{proof}
	We recall that $G_{k}$ is the multigraph obtained after $k$ steps of the Karp--Sipser algorithm applied to the random multigraph $\mathbb{G}(n,m_n)$. We will show that if $\eta$ is small enough, then with high probability, all the loops and multiple edges in $G_{k_0(\eta)}$ are either a loop on a vertex of degree $2$ or two edges joining the same pair of vertices of degree $2$. In particular, loops and multiple edges do not have leaves in their connected component, so they will never be removed by the Karp--Sipser algorithm.
	
	Let $\eta>0$ (we will fix its value later). For the sake of brevity, we will just write $k_0$ for $k_0(\eta)$. We denote by $\mathcal{E}_{\eta}$ the event that $|\At_{k_0}|, |\Bt_{k_0}|, |\Ct_{k_0}| \leq \log n$. We note that by Lemma~\ref{lem:fluctu}, the probability of $ \mathcal{E}_{\eta}$ tends to $1$ as $n \to +\infty$. On the event $\mathcal{E}_{\eta}$, Lemma~\ref{lem_XVH_not_too_small} ensures that the assumptions of Lemma~\ref{lem:degrees_are_Poisson} are satisfied, so we can write
	\begin{align*}  &\mathbb{E} \left[ \# \{\mbox{loops of $G_{k_0}$ attached to a vertex of degree at least $3$} \} | \Ff_{k_0} \right ] \\
	&\underset{{\color{white}\text{Lem.~6}}}{=} \sum_{j = X_{k_0}+1}^{X_{k_0}+V_{k_0}} \sum_{ i \geq 3} \mathbb{P} \left( \mathrm{deg}_{G_{k_0}}(j) = i \Big| \Ff_{k_0} \right) \frac{i(i-1)}{X_{k_0} + 2 V_{k_0} + S_{k_0}-1} \\
	&\underset{\text{Lem.~\ref{lem:degrees_are_Poisson}}}{=} \frac{V_{k_0}}{X_{k_0} + 2 V_{k_0} + S_{k_0}-1}  \left( \sum_{i = 3}^{\log (n)} i (i-1) \times O \left( \frac{Z_{k_0}^i}{i! f(Z_{k_0})} \right) + O \left( n^{-8} \right) \right),
	\end{align*}
	where the implied constants are absolute. By writing $\frac{Z_{k_0}^i}{f(Z_{k_0})} =O \left( Z_{k_0}^{i-2} \right) = O \left( Z_{k_0} C^{i-3} \right)$ where $C$ is an upper bound on $Z$, we find that the last display is $O \left( Z_k \right) + O \left( n^{-8} \right)$. Finally, if $\mathcal{E}_{\eta}$ occurs, then by Lemma~\ref{lem_XVH_not_too_small} we have $Z_{k_0}=O \left( \sqrt{\eta} \right)$. Therefore, there is an absolute constant $c$ such that for $n$ large enough, we have
	\[ \P \left( \mbox{$G_{k_0}$ has a loop attached to a vertex with degree at least $3$} \right) \leq c \sqrt{\eta}.\]
	We now prove the same with multiple edges. On the event $\mathcal{E}_{\eta}$, we can write
	\begin{align*} & \mathbb{E} \left[ \# \{\mbox{pairs of multiple edges of $G_{k_0}$ adjacent to a vertex of degree at least $3$} \} | \Ff_{k_0} \right] \\
	&\hspace{0.3cm}= \sum_{\substack{j_1,j_2 = X_{k_0}+1\\ j_1 \ne j_2}}^{X_{k_0}+V_{k_0}} \sum_{ \substack{i_1 \geq 3\\ i_2 \geq 2}} \mathbb{P} \left( \mathrm{deg}_{G_{k_0}}(j_1) = i_1, \mathrm{deg}_{G_{k_0}}(j_2) = i_2 \Big| \Ff_{k_0} \right) \frac{i_1(i_1-1) \cdot i_2(i_2-1)}{(X_{k_0} + 2 V_{k_0} + S_{k_0}-1)(X_{k_0} + 2 V_{k_0} + S_{k_0}-3)} \\
	&\underset{\text{Lem.~\ref{lem:degrees_are_Poisson}}}{=} \frac{V_{k_0}(V_{k_0}-1)}{(X_{k_0} + 2 V_{k_0} + S_{k_0}-1)(X_{k_0} + 2 V_{k_0} + S_{k_0}-3)} \left( O \left( n^{-4} \right) + \sum_{\substack{3 \leq i_1 \leq \log(n)\\2 \leq i_2 \leq \log(n)}} \frac{i_1 (i_1-1) i_2(i_2-1)}{i_1! i_2!} O \left( \frac{Z_{k_0}^{i_1} Z_{k_0}^{i_2}}{f(Z_{k_0})^2} \right) \right)\\
	&\hspace{0.3cm}= \ O \left( n^{-4} \right) + O(Z_k).
	\end{align*}
	Therefore, we have proved that the probability that $G_{k_0}$ has a loop or multiple edges incident to a vertex of degree at least $3$ is bounded by $c\sqrt{\eta}$ for some absolute constant $c$. This proves the Lemma by taking $\eta = \left( \frac{\delta}{c} \right)^2$.
\end{proof}

\begin{proof}[Proof of Theorem~\ref{thm:main}]
	We first prove the result for the model $ \mathrm{G}(n,m_n)$ with a fixed number $m_n$ of edges, where $m_n=\frac{\e}{2}n+O \left( \sqrt{n} \right)$.  Passing from this model to $G \left[ n, \frac{\mathrm{e}}{n} \right]$ will follow easily by integrating over the possible values of the number of edges of $G \left[ n, \frac{\mathrm{e}}{n} \right]$.
	
	Most of the proof will consist in proving the convergence of vertex degrees~\eqref{eqn:main_degree_convergence}. To keep the notation light, we will write down the proof completely only for vertices of degree $5$.
	
	Let $g$ be a bounded, continuous function from $\R$ to $\R$. For $ x >0$, we also write
	\[ \widehat{g}( x ) = \E \left[ g \left( \mathrm{Poi}\left( \frac{48}{5} x^{-5} \right) \right) \right],\]
	where $\mathrm{Poi}(\lambda)$ denotes a Poisson variable with parameter $\lambda$. We note that $\widehat{g}$ is bounded and is a continuous function of $\theta$. We also recall that for $i \geq 2$, we denote by $\D_i(n)$ the number of vertices of degree $i$ in the Karp--Sipser core of $\mathbb{G}(n,m_n)$. We want to prove that
	\begin{equation}\label{eqn:final_goal}
		\frac{\E \left[ g \left( \D_5(n) \right) \mathbbm{1}_{\mbox{$\mathbb{G}(n,m_n)$ is simple}} \right]}{\P \left( \mbox{$\mathbb{G}(n,m_n)$ is simple} \right)} \xrightarrow[n \to +\infty]{} \E \left[ \widehat{g}(\vartheta) \right],
	\end{equation}
	where we recall that $\vartheta$ is defined by~\eqref{eqn:defn_vartheta}. We first notice that by the Poisson approximation of Lemma~\ref{lem:degrees_are_Poisson} (for $k=0$) and by Lemma~\ref{lem:proba_simple}, the denominator of~\eqref{eqn:final_goal} is bounded away from $0$ (it actually goes to $\mathrm{e}^{-\frac{\mathrm{e}}{2}-\frac{\mathrm{e}^2}{4}}$ but we will not need the exact computation). We now study the numerator. We fix $ \delta >0$. Let $\eta>0$ small be given by Lemma~\ref{lem:intermediate_simplicity}, and let $K=K_{\eta,\delta}$ be the constant provided by Lemma~\ref{lem:fluctu}. We recall that $\mathcal{E}_{\eta, K}$ is the event that $\theta^n > k_0(\eta)$ and $|\At_{k_0(\eta)}|, |\Bt_{k_0(\eta)}|, |\Ct_{k_0(\eta)}| \leq K$, and that $\P \left( \mathcal{E}_{\eta,K} \right) \geq 1-\delta$ for $n$ large enough by Lemma~\ref{lem:fluctu}. We also write $\mathcal{E}'_{\eta}=\simpbis$. Finally, we write
	\[ \mathcal{E}'' = \left\{  \left| \frac{\sum_{i \geq 1} i^2 \D_i(n)}{4 \sum_{i \geq 1} i \D_i(n)} - \frac{1}{2} \right| \leq n^{-1/10} \right\}.\]
	Note that as a consequence of  Theorem~\ref{thm:multi}, we have $\sum_{i \geq 2} i \D_i(n)= 2 \D_2(n)\left( 1+O(n^{-1/5}) \right) $ and $\sum_{i \geq 2} i^2 \D_i(n)=4 \D_2(n)\left( 1+O(n^{-1/5}) \right) $ in probability, so $\P \left( \mathcal{E}'' \right) \to 1$ as $n \to +\infty$. Combining these observations, for $n$ large enough, we have
	\[ \P \left( \left( \mathcal{E}_{\eta,K} \cap  \mathcal{E}'_{\eta} \cap \mathcal{E}'' \right)^c \right) \leq 3 \delta. \]
	Therefore, we can write
	\begin{multline*}
		\E \left[ g(\D_5(n)) \mathbbm{1}_{\mbox{$\mathbb{G}(n,m_n)$ is simple}} \right] \\= O \left( \delta \right) + \E \left[ \mathbbm{1}_{\simp} \mathbbm{1}_{\mathcal{E}_{\eta,K}} \mathbbm{1}_{\mathcal{E}'_{\eta}} g(\D_5(n)) \mathbbm{1}_{\mathcal{E}''} \mathbbm{1}_{\mbox{$\KSC(\mathbb{G}(n,m_n))$ is simple}} \right],
	\end{multline*}
	where the constant implied by the $O$ notation only depends on the function $g$. Since $\P \left( \mathcal{E}'_{\eta} \right) \geq 1-\delta$ and everything inside of the expectation is bounded, we can remove the indicator $\mathbbm{1}_{\mathcal{E}'_{\eta}}$ and absorb the difference in the $O(\delta)$ term. By conditioning on $\Ff_{\theta^n}$ and on the degrees in the Karp--Sipser core, we obtain
	\begin{multline*}
	\E \left[ g(\D_5(n)) \mathbbm{1}_{\mbox{$\mathbb{G}(n,m_n)$ is simple}} \right] \\=  O \left( \delta \right) + \E \left[ \mathbbm{1}_{\simp} g(\D_5(n)) \mathbbm{1}_{\mathcal{E}_{\eta,K}} \mathbbm{1}_{\mathcal{E}''} \P \left( \mbox{$\KSC(\mathbb{G}(n,m_n))$ is simple} | \Ff_{\theta^n}, (\D_i(n))_{i \geq 2} \right) \right].
	\end{multline*}
	By Lemma~\ref{lem:proba_simple}, under the event $\mathcal{E}''$, the conditional probability for $\KSC(\mathbb{G}(n,m_n))$ to be simple is $ \mathrm{e}^{-3/4}+o(1)$, where $o(1)$ is uniform. Again, the error term can be absorbed in the $O(\delta)$ term, and we can then remove the indicator $\mathbbm{1}_{\mathcal{E}''}$ as we did before for $\mathbbm{1}_{\mathcal{E}'_{\eta}}$. By conditioning on $\Ff_{k_0(\eta)}$, the last display rewrites
	\begin{align*}
		O(\delta)+ \mathrm{e}^{-3/4} \E \left[ \mathbbm{1}_{\simp} \mathbbm{1}_{\mathcal{E}_{\eta,K}} \E \left[ g(\D_5(n)) | \Ff_{k_0(\eta)} \right] \right].
	\end{align*}
	By the Markov property, on the event $\theta^n>k_0(\eta)$ (which is implied by $\mathcal{E}_{\eta,K}$), the conditional expectation $\E \left[ g(\D_5(n)) | \Ff_{k_0(\eta)} \right]$ is a measurable function of $\left( \At, \Bt, \Ct \right)_{k_0(\eta)}$. Moreover, Theorem~\ref{thm:bis} guarantees that on $\mathcal{E}_{\eta,K}$, this quantity converges towards $\mathbb{E} \left[ g(\widehat{\vartheta}) \right]$, uniformly in the values $\left( \At, \Bt, \Ct \right)_{k_0(\eta)}$. Therefore, for $n$ large enough, we have
	\[ \E \left[ g(\D_5(n)) \mathbbm{1}_{\mbox{$\mathbb{G}(n,m_n)$ is simple}} \right] = O(\delta) +  \mathrm{e}^{-3/4} \E \left[ \widehat{g}(\vartheta) \right] \P \left( \simp \mbox{ and } \mathcal{E}_{\eta,K} \right). \]
	As before, using this time Lemma~\ref{lem:fluctu} and the choice of the constant $K$, the event $\mathcal{E}_{\eta,K}$ can be absorbed by the $O(\delta)$ term. Coming back to~\eqref{eqn:final_goal} and using that $\P \left( \mbox{$\mathbb{G}(n,m_n)$ is simple} \right)$ is bounded away from $0$, we have shown that for any $\delta>0$, there is $\eta>0$ such that for $n$ large enough, we have
	\begin{equation}\label{eqn:independence_degrees_simplicity}
		\E \left[ g(\D_5(n)) | \mbox{$\mathbb{G}(n,m_n)$ is simple} \right] = O(\delta) + \frac{ \mathrm{e}^{-3/4} \P \left( \simp \right)}{\P \left( \mbox{$\mathbb{G}(n,m_n)$ is simple} \right)} \E \left[ \widehat{g}(\vartheta) \right],
	\end{equation}
	where the constant implied by the $O$ notation only depends on $g$.	In particular, if $g$ is constant and equal to $1$, then so is $\widehat{g}$ and the last formula becomes
	\[ \P \left( \simp \right) = \left( 1+O(\delta) \right) \mathrm{e}^{3/4} \P \left( \mbox{$\mathbb{G}(n,m_n)$ is simple} \right).\]
	Plugging this back into~\eqref{eqn:independence_degrees_simplicity}, we get
	\[ \E \left[ g(\D_5(n)) | \mbox{$\mathbb{G}(n,m_n)$ is simple} \right] = \E \left[ \widehat{g}(\vartheta) \right] + O(\delta).\]
	Since we can take $\delta$ arbitrarily small, this concludes the proof of convergence of $\D_5(n)$ for the random simple graph $ \mathrm{G}(n,m_n)$.
	
	To extend the result to $\left( D_i(n) \right)_{i \geq 2}$ instead of just $D_5(n)$, the proof is exactly the same. We just need to replace $g(\D_5(n))$ by 
	\[ g \left( \frac{\D_2(n)}{n^{3/5}}, \frac{\D_3(n)}{n^{2/5}}, \frac{\D_4(n)}{n^{1/5}}, \D_5(n), \sum_{i \geq 6} \D_i(n) \right),\]
	where $g$ is a bounded, continuous function from $\R^5$ to $\R$.
	
	We now move on to the claim that conditionally on its vertex degrees, the Karp--Sipser core of $ \mathrm{G}(n,m_n)$ is a configuration model conditioned to be simple. This will be true without assuming anything about $m_n$. For this, let $ \mathfrak{g}$ be a simple graph with no leaf and no isolated vertex. We need to show that $\P \left( \KSC( \mathrm{G}(n,m_n))= \mathfrak{g}\right)$ only depends on the vertex degrees of $ \mathfrak{g}$. But we have
	\[ \P \left( \KSC( \mathrm{G}(n,m_n))= \mathfrak{g} \right) = \frac{\P \left( \mbox{$\mathbb{G}(n,m_n)$ is simple and $\KSC(\mathbb{G}(n,m_n))= \mathfrak{g}$ } \right)}{\P \left( \mbox{$\mathbb{G}(n,m_n)$ is simple} \right)}. \]
	The denominator does not depend on $ \mathfrak{g}$, so we only need to consider the numerator. Since $ \mathfrak{g}$ is simple, it can be rewritten as
	\begin{multline*}
	\P \left( \mbox{$\simp$ and $\simpbis$ and $\KSC(\mathbb{G}(n,m_n))= \mathfrak{g}$}  \right)\\
	= \E \left[ \mathbbm{1}_{\simp} \mathbbm{1}_{\simpbis} \P \left( \KSC(\mathbb{G}(n,m_n))= \mathfrak{g} | \Ff_{\theta^n} \right) \right].
	\end{multline*}
	But by Lemmas~\ref{lem:markov} and~\ref{lem:config_model}, the core $\KSC(\mathbb{G}(n,m_n))$ is a configuration model conditionally on $\Ff_{\theta^n}$ and on its vertex degrees, so this only depends on the vertex degrees of $ \mathfrak{g}$.
	
	Finally, to pass from $ \mathrm{G}(n,m_n)$ to the ``true" Erd\H{o}s--R\'enyi model $ \mathrm{G}[n, \frac{\mathrm{e}}{n}]$, let $M_0^n$ be the total number of edges of $ \mathrm{G}[n, \frac{\mathrm{e}}{n}]$. By the Skorokhod embedding theorem, we may assume that $\frac{M_0^n-\mathrm{e}n/2}{\sqrt{n}}$ converges a.s. to a Gaussian variable, so in particular $M_0^n=\frac{\e}{2}n+O \left( \sqrt{n} \right)$. Therefore, if $g:\R^5 \to \R$ is bounded and continuous, we have the almost sure convergence
	\[ \E_{ \mathrm{G}[n, \frac{\mathrm{e}}{n}]} \left[ g \left( \frac{\D_2(n)}{n^{3/5}}, \dots, \sum_{i \geq 6} \D_i(n) \right) \Big| M_0^n \right] \xrightarrow[n \to +\infty]{a.s.} \mathbb{E}[\widehat{g}(\vartheta)],\]
	and~\eqref{eqn:main_degree_convergence} for $ \mathrm{G}[n, \frac{\mathrm{e}}{n}]$ follows by taking the expectation. Finally, for any simple graph $ \mathfrak{g}$, we can write the decomposition
	\[ \P \left( \KSC ( \mathrm{G}[n,p])= \mathfrak{g} \right) = \sum_{m=0}^{\binom{n}{2}} \P \left( M_0^n=m \right) \P \left( \KSC( \mathrm{G}(n,m))= \mathfrak{g} \right), \]
	which only depends on $ \mathfrak{g}$ through its vertex degrees. It follows that conditionally on its vertex degrees, the graph $\KSC ( \mathrm{G}[n,p])$ is a configuration model conditioned to be simple.
\end{proof}

\bibliographystyle{siam}
\bibliography{bibli}


\end{document}